\documentclass[a4paper,10pt]{article}
\usepackage[toc,title]{appendix}
\usepackage{amsthm}
\usepackage[a4paper, total={6in, 8in}]{geometry} % wider margins
\usepackage{authblk} % authors
\usepackage{booktabs}  % for '\midrule' macro
\usepackage[section]{placeins} % force figures into their sections
\renewenvironment{flushright}{%
    \begin{center}%
}{%
    \end{center}%
    \ignorespacesafterend%
}

\usepackage{graphicx}      % include this line if your document contains figures
\graphicspath{{./tikz/}}
\usepackage{xcolor}
\usepackage[utf8]{inputenc}

\usepackage{algorithm}
\usepackage{algpseudocode}

\usepackage{subfig}

\usepackage{titlesec}
\titlespacing*{\subsubsection}{0pt}{10pt}{0pt}
\usepackage[]{placeins}
\let\Oldsubsection\subsection
\renewcommand{\subsection}{\FloatBarrier\Oldsubsection}

\usepackage{changepage}
\usepackage{svg}
\usepackage{tikz}
\usepackage{booktabs}
\usepackage{pgfplots}
\usepackage{tikzscale}
\usepackage{pgfkeys}
\pgfplotsset{compat=newest}
\newenvironment{bmatrixcolor}[1][red]
  {\colorlet{savethecolor}{.}\colorlet{bracecolor}{#1}%
    \color{bracecolor}\left[\color{savethecolor}\begin{matrix}}
  {\end{matrix}\color{bracecolor}\right]}
\usetikzlibrary{arrows.meta}
\usetikzlibrary{backgrounds}
\usepgfplotslibrary{patchplots}
\usepgfplotslibrary{fillbetween}
\pgfplotsset{%
layers/standard/.define layer set={%
    background,axis background,axis grid,axis ticks,axis lines,axis tick labels,pre main,main,axis descriptions,axis foreground%
}{grid style= {/pgfplots/on layer=axis grid},%
    tick style= {/pgfplots/on layer=axis ticks},%
    axis line style= {/pgfplots/on layer=axis lines},%
    label style= {/pgfplots/on layer=axis descriptions},%
    legend style= {/pgfplots/on layer=axis descriptions},%
    title style= {/pgfplots/on layer=axis descriptions},%
    colorbar style= {/pgfplots/on layer=axis descriptions},%
    ticklabel style= {/pgfplots/on layer=axis tick labels},%
    axis background@ style={/pgfplots/on layer=axis background},%
    3d box foreground style={/pgfplots/on layer=axis foreground},%
    },
}
\pgfplotsset{ /tikz/every picture/.append style={trim axis left, trim axis right}}

\usepgfplotslibrary{external} 
\usepackage{environ}
\makeatletter
\newsavebox{\measure@tikzpicture}
\NewEnviron{scaletikzpicturetowidth}[1]{%
  \def\tikz@width{#1}%
  \begin{lrbox}{\measure@tikzpicture}%
  \BODY
  \end{lrbox}%
  \pgfmathparse{#1/\wd\measure@tikzpicture}%
  \BODY
}
\makeatother
\newtheorem{lemma}{Lemma}
\newtheorem{corollary}{Corollary}
\newtheorem{example}{Example}

\usepackage{amsmath,amssymb,amsfonts}
\usepackage{textcomp}
\usepackage{printlen}
\usepackage{mathtools}
\usepackage{placeins}
\usepackage{siunitx}
\usepackage{bm}
\usepackage{nicefrac}

\newcommand{\vect}[1]{\boldsymbol{\mathbf{#1}}}

\makeatletter
\newcommand\deflabel[1]{\def\@currentlabel{#1}}

\newcommand{\pushright}[1]{\ifmeasuring@#1\else\omit\hfill$\displaystyle#1$\fi\ignorespaces}
\makeatother

\usepackage{algorithm}
\usepackage{amsmath}

\DeclareMathOperator*{\argmin}{arg\!\min}

\usepackage{hyperref}

\DeclareMathOperator*{\minimize}{minimize}

\makeatletter
\newcommand{\bigcomp}{%
  \DOTSB
  \mathop{\vphantom{\sum}\mathpalette\bigcomp@\relax}%
  \slimits@
}
\newcommand{\bigcomp@}[2]{%
  \begingroup\m@th
  \sbox\z@{$#1\sum$}%
  \setlength{\unitlength}{0.9\dimexpr\ht\z@+\dp\z@}%
  \vcenter{\hbox{%
    \begin{picture}(1,1)
    \bigcomp@linethickness{#1}
    \put(0.5,0.5){\circle{1}}
    \end{picture}%
  }}%
  \endgroup
}
\newcommand{\bigcomp@linethickness}[1]{%
  \linethickness{%
      \ifx#1\displaystyle 2\fontdimen8\textfont\else
      \ifx#1\textstyle 1.65\fontdimen8\textfont\else
      \ifx#1\scriptstyle 1.65\fontdimen8\scriptfont\else
      1.65\fontdimen8\scriptscriptfont\fi\fi\fi 3
  }%
}

\newcommand{\R}{\mathbb{R}}

\makeatother

\newcommand{\inputnamedtex}[1]{%
	    \includegraphics{tikz/#1-eps-converted-to.pdf}% # use second, one you have the .eps files
}
\title{\LARGE \bf
Dictionary-free Koopman model predictive control with nonlinear input transformation
}

\author[1,2]{Vít Cibulka}
\author[1,2]{Milan Korda}
\author[1]{Tomáš Haniš}

\affil[1]{Department of Control Engineering, Faculty of Electrical Engineering,
Czech Technical University in Prague, The Czech Republic

{\tt\small vit.cibulka@fel.cvut.cz, tomas.hanis@fel.cvut.cz}}
\affil[2]{CNRS, Laboratory for Analysis and Architecture of Systems, Toulouse, France
 {\tt\small korda@laas.fr}
}
\definecolor{dkred}{rgb}{0.0,0,0}
\definecolor{dkgreen}{rgb}{0,0.0,0}
\definecolor{corrige}{rgb}{0,0,0}
\definecolor{corrige2}{rgb}{0,0,0}
\definecolor{corrige3}{rgb}{0,0,0.0}
\definecolor{color_MKres}{rgb}{0,0.0,0.0}
\definecolor{corrigeplane}{rgb}{0,0,0}
\definecolor{vclast}{rgb}{0,0.0,0}

\newcommand{\newmk}[1]{{\color{dkgreen}{#1}}}
\newcommand{\COR}[1]{{\color{corrige}{#1}}}
\newcommand{\newcor}[1]{{\color{corrige2}{#1}}}
\newcommand{\corplane}[1]{{\color{corrigeplane}{#1}}}
\newcommand{\vclast}[1]{{\color{vclast}{#1}}}

\newcommand{\corr}[1]{{\color{corrige3}{#1}}}

\definecolor{imb}{rgb}{0,0,1}

\newcommand{\exampleskip}
    {
    \vskip 13pt
    }

\begin{document}

\maketitle
\thispagestyle{plain}
\let\thefootnote\relax\footnotetext{$^*$This work has been supported 
by the Czech Science Foundation (GACR) under contract No. 20-11626Y, 
and by
the Grant Agency of the Czech Technical University in Prague, grant No. SGS22/166/OHK3/3T/13.
This work has also been supported
by the AI Interdisciplinary Institute ANITI funding, through the
French “Investing for the Future PIA3” program under the Grant agreement n$^\circ$ ANR-19-PI3A-0004 as well as by the National Research Foundation, Prime Minister’s Office, Singapore, under its Campus for Research Excellence and Technological Enterprise (CREATE) programme.\\
$^{**}$ This work has been submitted to SIAM Journal on Applied Dynamical Systems (SIADS) and is currently under review.
}
% \let\thefootnote\relax\footnotetext{A footnote without numbering}
% \let\thefootnote\relax\footnotetext{A footnote without numbering}

% \pagestyle{empty}

% REQUIRED
\begin{abstract}
  This paper introduces a method for data-driven control based on the Koopman operator model predictive control. Unlike existing approaches, the method does not require a dictionary and incorporates a nonlinear input transformation, thereby allowing for more accurate predictions with less ad hoc tuning. In addition to this, the method allows for input quantization and exploits symmetries, thereby reducing computational cost, both offline and online. Importantly, the method retains convexity of the optimization problem solved within the model predictive control online. Numerical examples demonstrate superior performance compared to existing methods as well as the capacity to learn discontinuous lifting functions.
\end{abstract}

% REQUIRED
% \begin{keywords}
% Koopman operator, Nonlinear control, Model predictive control,
% Linear predictors, Data-driven control design, Lifting,
% Dictionary, Symmetry
% \end{keywords}

% % REQUIRED
% \begin{MSCcodes}
%   47N70, 93C10, 37L20, 47B38, 93B45, 90C25, 93C10, 37N35
% \end{MSCcodes}
% 37L20  	Symmetries of infinite-dimensional dissipative dynamical systems
% 37N35  	Dynamical systems in control
% 93C10  	Nonlinear systems in control theory
% 90C25  	Convex programming
% 93B45  	Model predictive control
% 47N70  	Applications of operator theory in systems, signals, circuits, and control theory
% 47B38  	Linear operators on function spaces (general)

\section{Introduction}

The Koopman operator approach is a successful framework for data-driven analysis of nonlinear dynamical systems that originated in the 1930s in the seminal works \cite{Koopman1931,koopman1932dynamical} but gained wider popularity only much later in the mid 2000s with the advent of modern-day computing, starting with the works \cite{Mezic2005,mezic2004comparison}. The core idea is to represent the nonlinear system by an infinite-dimensional linear operator acting on the space of functions defined on the state-space (referred to as observables). Finite-dimensional approximations of this operator, computable from data using simple linear algebraic tools, then allow one to gain insight into the underlying dynamics and construct reduced-order models. This  approach can be thought of as lifting of the original nonlinear dynamics into a higher dimensional space where it admits a linear description.

Generalizing this concept to control systems is far from trivial. A promising approach proposed in \cite{Korda_koop} defined the Koopman operator with control as acting on the product of the state-space and the space of all control sequences and used finite-dimensional approximations of this operator within a linear model predictive control scheme, nowadays referred to as the Koopman MPC. The primary benefit is the convexity of the MPC problem solved online with computational complexity compared to classical linear MPC. The primary drawback is the inherent limitation of prediction accuracy. Indeed, the predictors proposed in \cite{Korda_koop} lift only the state whereas the control input remains untransformed and enters linearly into the predictor used within the MPC in order to preserve convexity. Therefore, fixing an initial condition, the mapping from the control input sequence to the output of the predictor is linear whereas the true mapping realized by the nonlinear system is typically nonlinear. A possible remedy is to consider bilinear predictors, which has a long history related to generalizing the classical Carleman linearization~\cite{carleman1932application} to a controlled setting; see, e.g., \cite{germani2005filtering,rauh2009carleman,armaou2014piece,surana2016koopman}. This, however, necessarily spoils the convexity of the MPC problem solved online and is hence not pursued in this work. We refer the reader to the survey~\cite{brunton2021modern} for further discussion regarding the use of the Koopman operator in control.

A second drawback, common in controlled and uncontrolled settings, is the need for a user-specified dictionary of nonlinear lifting functions. Although such dictionaries may be naturally available for measure-preserving ergodic systems in the form of time-delayed measurements of the state \cite{arbabi2017ergodic}, no such choice is available in more general settings and hence past works resorted to ad-hoc dictionaries (e.g., \cite{koop_edmd,Korda_koop}) such as radial basis functions or monomials. Even though efforts have been made to alleviate this by using a general parametrization of the dictionary (e.g., with neural networks \cite{li2017extended}) or by exploiting the dynamics and state-space geometry to construct them \cite{korda2020optimal}, the need for a significant engineering effort in the design process has not been eliminated. 

This work addresses the two principal issues described above as well as demonstrates how input quantization and symmetries of the problem can be exploited. Specifically, we first propose a class of predictors with nonlinear input transformation that preserves convexity of the MPC problem solved online and greatly improves the long-term prediction accuracy. To the best of our knowledge this is the first time that a nonlinear input transformation is employed while preserving convexity of the MPC problem.

Second, we eliminate the need for a dictionary by treating the values of the lifting functions on the available data samples as optimization variables in the (offline) learning process.  Concurrently with our work, a different dictionary-free method for construction of Koopman predictors was  developed in \cite{otto2022learning} using hidden Markov models; however, albeit providing accurate predictions, the models constructed therein are bilinear and hence lead to nonconvex problems when used within MPC.

In addition to this, we allow for input quantization, thereby bringing the method closer to practical applications where the control input is typically quantized.

Finally, we exploit symmetries in designing the predictors, thereby allowing for a more parsimonious parametrization of the predictors and and hence faster learning;  this could be seen as a generalization of the symmetry-exploitation methods without control proposed in \cite{salova2019koopman} to the controlled setting where the interplay between state-space and control-space symmetries must be carefully taken into account.

\vclast{We demonstrate the effectiveness of our approach on several examples, including the 
highly nonlinear vehicle model adopted from our previous work in \cite{Cibulka2021}.
The current research around the Koopman operator for vehicle control includes \cite{Marko1} and \cite{MarkoTV},
which both use the method \cite{Korda_koop} and therefore require user-specified dictionary for the state lifting. On the other hand, the work~\cite{Cibulka2021} does not require an explicit dictionary of lifting functions but 
it still involves a notable engineering effort to construct the state lifting. Our paper uses the same vehicle model as \cite{Cibulka2021}, while minimizing the effort to construct the state lifting and, for the first time, employs
input lifting. The proposed Koopman MPC  significantly outperforms the previous methods on this vehicle model and, interestingly, 
the analysis of the input lifting functions unveils certain physical properties that
can be directly tied to the vehicle dynamics.
}

% We test our algorithm on the highly nonlinear vehicle model.
% Some works that used Koopman with vehicles are 
% \cite{Cibulka2021} and \cite{Marko1}
%  and \cite{MarkoTV}
%  \cite{chineseKoopman}
% Marko with torque vectoring - not the same use case,
% Chinese guys with INNs - driven only in linearly behaving region.

\textbf{Structure of this paper}
\vclast{The Section \ref{sec:problem_statement} provides the problem statement and
the Section \ref{sec:finding_predictor} presents the main result of this paper.
Then we discuss symmetry exploitations in the Section \ref{sec:symmetry}. 
Various implementation details of the method are presented in Section~\ref{sec:implementation}. The usage of the Koopman operator in control
is presented in the Section \ref{sec:control}, and the concrete form of the Koopman MPC used in this paper is summarized in the Section \ref{sec:summary_mpc}.
The numerical examples are presented in the Section \ref{sec:numerical_examples} and we conclude in the Section \ref{sec:conclusion}.
}

\textbf{Notation}
\newcor{All vectors are assumed to be column vectors.}
The set of integers from \(1\) to \(n\) will be denoted as \(\mathbb{Z}_{1,n}\). 
The cardinality of a set \(A\) is denoted as \(|A|\).
The identity matrix of size \(n\) is \(I_{n}\); the column vector of ones 
of size \(n\) is denoted as \(\mathbf{1}_{n}\). 
The element-wise multiplication is denoted by~\(\odot\).
\newcor{A block-diagonal matrix with blocks $x_1,\ldots,x_n$ on the diagonal is denoted 
as $\textrm{bdiag}(x_1,\ldots,x_n)$. For a vector $x$ and a positive semidefinite matrix $Q$, we denote \(||x||_{Q} := \sqrt{x^{\top}Qx}\).
}

\section{Problem statement}
\label{sec:problem_statement}
Let us consider a discrete-time controlled nonlinear dynamical system
\begin{align}
\begin{split}
	\label{eq:nlsys}
	x_{t+1} &= f(x_t,u_t) \\
	y_t &= g(x_t),
\end{split}
\end{align}
where \(x_t \in X \subset \R^{n_x}\), \(u_t \in U \subset \R^{n_u}\),
\(y_t \in Y \subset \R^{n_y}\) are the state, the input, and the output vectors.

We define the Koopman operator \(\mathcal{K}\) with control input 
similarly as in \cite{Korda_koop}, 
by considering the extended state-space $\R^{n_x} \times \ell(U)$, where $\ell(U) := \{ (u_i)_{i=0}^{\infty} \mid u_i\in U\}$ is the space of all control sequences. We shall denote the elements of $\ell(U)$ by $ \vect{u} := (u_i)_{i=0}^{\infty}$. The extended dynamics is defined by
\begin{equation}
	\chi^+ = F(\chi) = \begin{bmatrix}
		f(x,\vect{u}(0)) \\ \mathcal{S}\vect{u}
	\end{bmatrix},
	\text{ for } \chi(0) = \begin{bmatrix}
		x_0 \\ \vect{u}_0
	\end{bmatrix},
\end{equation}
where \(\chi = (x,\vect{u}) \in \mathbb{R}^{n_x} \times \ell(U) \) is the extended state, \(\vect{u}(0)\) denotes the first element of the sequence \(\vect{u}\), and \(S\) is a shift operator such that \((S\vect{u})(i) = \vect{u}(i+1)\).
The Koopman operator $\mathcal{K} : \mathcal{H}\rightarrow \mathcal{H}$
%  $\mathcal{K} : \R^{n_x}\times \ell(U) \to \R^{n_x}\times \ell(U)$ 
 is then defined by
\begin{equation}
	(\mathcal{K}\xi)(\chi) := \xi(F(\chi))
\end{equation}
for \(\xi:\R^{n_x}\times \ell(U)\to \R\) that belongs 
to some space of observables \(\mathcal{H}\). The functions $\xi$ that the Koopman operator acts on are referred to as observables. \corr{In this work, we} shall assume the observables in the form 
\begin{equation}
	\xi(\chi) = \begin{bmatrix}
		\Phi(x) \\ \Psi(\vect{u}(0))
	\end{bmatrix}.
\end{equation}

% \MK{Mam cist i to modry v revise?}

% % \revise{
% \begin{equation}
% 	\left(\mathcal{K} \begin{bmatrix}
% 	 \Phi \\ \Psi
%    \end{bmatrix} \right) \left(\begin{bmatrix}
% 	 x \\ u
%    \end{bmatrix}\right)
%    = \Phi(f(x,u)),
%  \end{equation}
 
%  \MKres{Tohle nedava uplne smysl: jako vstup ten operator vezme $\Phi$ a $\Psi$, ale na vystupu se $\Psi$ vubec neobjevi. VC: 
%  naposled jsi mi rikal ze muzu min takhle a akorat zminit v textu ze maji jinej 
%  domain a codomain. Pokud ne toto tak me napada akorat napsat input-liftnutou verzi toho co mas v Automatice.} 
 
%  \MK{Tady odsud to uz neni presne. $\Phi$ a $\Psi$ nejsou jen nejake observables, ale vektory observables, ktere jsou specificky vybrane pro nasi aplikaci.}
 
 where \(\Phi: X \rightarrow Z\) is a \corr{vector} state lifting function, \mbox{\(\Psi: U \rightarrow V\)} is the input transformation.
 The spaces \(Z \subset \mathbb{R}^{n_z} \) and \(V \subset \mathbb{R}^{n_v}\) are the Koopman state and input spaces respectively.
 
%  Note that the domain and codomain of the operator do not 
%  have the same dimension. {\color{black}This is because 
%  we consider controlled nonlinear dynamics, unlike the usual 
%  definitions which are for autonomous systems only.}
Note that with this definition, the operator \(\mathcal{K}\) will predict 
the future lifted inputs; we will disregard those 
since we are not interested in the spectral
 properties of the operator. Our goal is linear prediction of the {\color{black}controlled} nonlinear dynamics \eqref{eq:nlsys}.

Furthermore, since \(\mathcal{K}\) is 
 generally infinite-dimensional, we will work with its finite-dimension 
 truncation in the form of an LTI predictor
 \begin{equation}
 \begin{alignedat}{2}
	\label{eq:koop_LTI}
	z_{t+1} &= Az_t + Bv_t &\\
	\hat{y}_t &= Cz_t, &\text{ for } z_0 = \Phi(x_0) \text{ and } v_t = \Psi(u_t),
 \end{alignedat}
 \end{equation}
 
 where \newmk{$A\subset \R^{n_z\times n_z}$, $B\subset \R^{n_z\times n_v}$, $C\subset \R^{n_y\times n_z}$} are matrices to be determined. \newmk{The system~(\eqref{eq:koop_LTI}) is referred \newcor{to} as the \emph{Koopman predictor}; the state \(z_t\) is referred to as the \emph{lifted state}} and \(\hat{y}_t\) is the \newmk{predicted} output of the system \eqref{eq:nlsys}. 
 % \MK{At some point, you need to write that in this paper $n_v = n_u$.}

%  }%color

The primary goal of this work is to find \(A,B,C,\Psi,\Phi\), 
\corplane{such that the resulting LTI system \eqref{eq:koop_LTI}  predicts the behaviour of the nonlinear dynamics \eqref{eq:nlsys} on \(X\) for \(H_{\rm T}\) steps ahead. This predictor is then used for controller synthesis within Koopman MPC,
which we describe in Section \ref{sec:control}.}
% \corplane{such that the predicted outputs $\hat{y}_t$ of \eqref{eq:koop_LTI} are as close as possible to the real outputs $y_t$ of \eqref{eq:nlsys} on the set \(X\).}
% \COR{
% such that for some trajectory of length \(H_{\rm T}\), the norm
% % \MK{Pro jaky $k$ ma byt ta norma minimalni? Beres average pres najaky horizont?}
% % \begin{equation}
% % 	\label{eq:norm}
% % 	|| \hat{y}_k(z,v) - y_k(x,u)||	
% % \end{equation} 
% \begin{equation}
% 	\label{eq:norm}
% 	\sum_{t=0}^{H_{\rm T}} || \hat{y}_t(z,v) - y_t(x,u)||_2^2
% \end{equation} }
% \MK{ TODO I suggest changing the notation from $H_T$ (=trajectory length) to $N$ and from $N$ (= number of trajectories) to $M$. This is consistent with my previous works. This will mess up the whole paper so be careful when changing it.}
% is minimal for
% \(x \in X, u \in U, v \in V, z \in Z\), where \(X,U,V\), and \(Z\)
% are the original state space, original input space, transformed input space, and lifted state space 
% in this order. 
% \MK{Equation~\eqref{eq:norm} does not make sense. I am sure you can figure out why.}

\section{Finding the Koopman predictor} 
\label{sec:finding_predictor}
In this section we present a method for determining the Koopman predictor~(\eqref{eq:koop_LTI}). We start by giving an intuition: Assuming that we have trajectory data \((y_t,u_t,x_0)\) \corr{for times \(t \in \mathbb{Z}_{0,H_{\rm T}} \)}, 
% \MK{Bud vic precizni, mas \((y_t,u_t,x_0)\) pro jedno $t$ nebo vic?} 
we want to find the linear system \eqref{eq:koop_LTI} so that 
its output \(\hat{y}_t\) matches the trajectory data  \(y_t\). The distinguishing feature of our work is that we do so in a \emph{dictionary-free} manner, that is, by optimizing the values of \(\Phi\) and \(\Psi\) on the available data  without having to specify a dictionary of basis functions parametrizing \(\Phi\) and \(\Psi\). In addition to optimizing the values of $\Phi$ and $\Psi$, we also optimize over the system matrices $A$, $B$, $C$. See Figure~\ref{fig:koop_expl} for an illustration.

% We do this by searching over \textit{all} possible parameters
% of the Koopman operator: the system matrices \(A\),
% \(B\), \(C\), and the lifting functions \(\Phi\) and \(\Psi\).
%  The situation is depicted in Fig.\ref{fig:koop_expl}. 
%  \MK{Tohle je hodne basic intuice a vubec to neni to hlavni o cem je tehle paper. V tomhle paperu je ta zakladni myslenka, ze nemas explicitni parametrizaci $\Phi$ a $\Psi$, ale optimalizujes pres jejich hodnoty na samplech.}

\begin{figure}[!ht]
	\centering
	\includegraphics{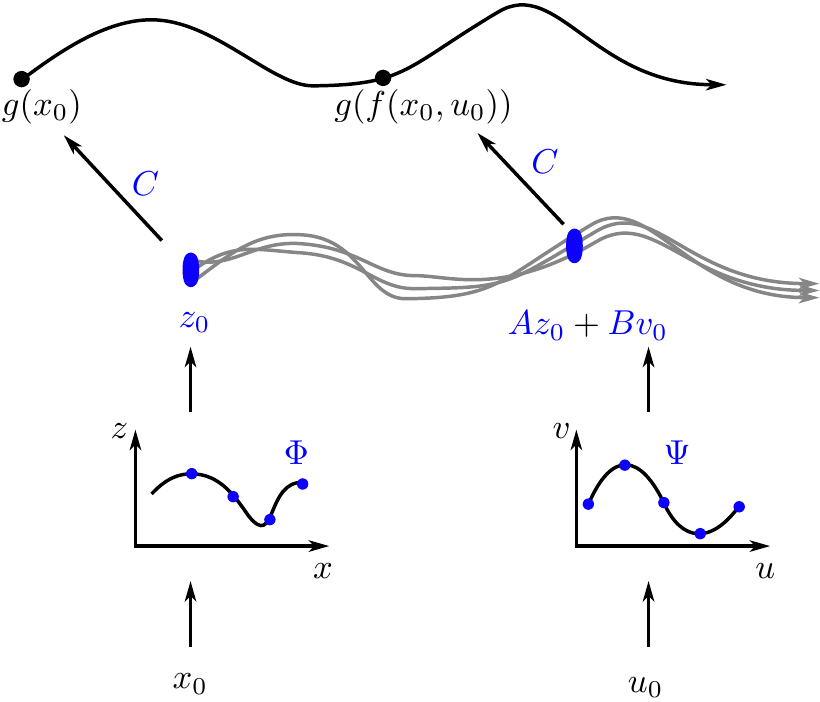}
	\caption{
		The intuition behind the Koopman predictor with control. 
		Both the initial state and the inputs are lifted via the functions 
		\(\Phi\) and \(\Psi\) respectively.
		The linearly evolving high-dimensional trajectory 
		of the Koopman predictor \eqref{eq:koop_LTI} can then be projected onto the original state 
		space, giving us the nonlinear trajectory of \eqref{eq:nlsys}.
		The quantities optimized in our approach are highlighted in blue; this includes the values of \(\Phi\) and \(\Psi\) on the available data. 
		% \MK{$A,B,C, z_0,v_0$ atd by taky meli byt modre.} 
  % \MK{In the figure, it should be $x$ not $u$ on the axis for $\Phi$, right?}
  }
		\label{fig:koop_expl}
\end{figure}	

\subsection*{Basic definitions}
\label{sec:basic_def}
The optimization will be done over
\(N\) trajectories of the \corr{nonlinear} system \eqref{eq:nlsys}.
% ,
% where the trajectories are chosen so that the output set \(Y\) is covered with a sufficient density. 
% \MKres{Strange? Perhaps erase. VC: what should I say about selection of the trajectories then?}
% \MK{Nic o tom nerikej. Jen zacni, ``Assume we have $N$ trajectories of the form...'' a rovnou definuj precizne tvuj data set.}
\corr{To define the trajectories,}
we first need to define the sets of initial conditions
and admissible system inputs.

We will assume to have samples \(x^i_0 \) from \(X\), which are the initial states of the aforementioned trajectories of 
\eqref{eq:nlsys}
\begin{equation}
	X_0 = \{x^i_0 \in X :  i \in \mathbb{Z}_{1,N}\}.
\end{equation}
\COR{To each $x^i_0 \in X$, we associate a lifted initial condition \mbox{$z_0^i \in Z_0 \newcor{\subset \mathbb{R}^{n_z}}$}; these will be decision variables in the optimization problem where the control Koopman predictor is learned.} 

We will consider each input channel of the system \eqref{eq:nlsys} individually 
and assume that it is normalized to $[-1,1]$.
\corplane{We also permit quantization of the control signal,
thereby taking into account a digital control implementation or control signals that are discrete in nature  (e.g., gears in a vehicle).} 
The \(k^{\rm th}\) input channel with \(q_k\) quantization levels reads
\begin{equation}
	\label{eq:channelwise}
	U_k \subset [-1, 1],\;\; |U_k| = q_k.
\end{equation}
The full input space can be retrieved as the cartesian 
product of the individual channels
\begin{equation}
	\label{eq:cartprod}
	U = U_1 \times U_2 \times \dots \times U_{n_u}.
\end{equation}

In the optimization process, we will be searching for the  transformed input channels
 \(V_k\), which have the same number 
of quantization levels as \(U_k\):
\begin{equation}
	\label{eq:channelwise_V}
	V_k \subset [V_{k,\rm{min}}, V_{k,\rm{max}}],\;\; |V_k| = q_k.
\end{equation}
We do not make assumptions on the maxima and minima of the lifted 
input channels, since they are optimization variables.

After finding the Koopman predictor, we can retrieve the lifted input space as
\begin{equation}
	\label{eq:Vcart}
	V = V_1 \times V_2 \times \dots \times V_{n_u}.
\end{equation}
\newcor{We will assume that the dimensions of both the original and lifted inputs are the same, i.e. $n_v = n_u$. Although this assumption is not required by the algorithm itself, the benefit of raising the 
input dimension is  \corplane{not investigated} in this paper. 
The values $V_k$ are decision variables to optimize.

% the values used in this work 
% are the same as the values $U_k$, as mentioned in the Table \ref{table:init}.
}

% We would like to stress that the full input space \(U\)
The Example \ref{ex:quantization} shows the advantage of 
considering the inputs channel-by-channel, instead of working with the 
whole cartesian product \newcor{\eqref{eq:cartprod}}.

\begin{example}
\label{ex:quantization}
Assume that we have two input channels quantized as
\begin{align}
\begin{split}
	\label{eq:ex_sum}
\newcor{U_1 = \{-1,1\},U_2 = \{-1,0.4,1\},}
\end{split}
\end{align}
with \(q_1 = 2\) and \(q_2 = 3\). 
% \MK{Why do you need the subscript for the sets?}
We retrieve \(U\) as
\begin{equation}
	\label{eq:ex_cart}
	U = \left\{
	\begin{bmatrix}
		-1\\-1
	\end{bmatrix},
	\begin{bmatrix}
		1\\-1
	\end{bmatrix},\begin{bmatrix}
		-1\\0.4
	\end{bmatrix},\begin{bmatrix}
		1\\0.4
	\end{bmatrix},\begin{bmatrix}
		-1\\1
	\end{bmatrix},\begin{bmatrix}
		1\\1
	\end{bmatrix}
	\right\}.
\end{equation}
% For simplicity, assume that the lifted inputs \(V_k\) have the same dimensions as \(U_k\), so the sets \(V_k\) and \(V\) will 
% have identical number of elements as \(U_k\) and \(U\) respectively.

Optimizing over the channels \(V_1\) and \(V_2\) will introduce 
\(\sum_k q_k = 5\) optimization variables.
Should we, however, optimize over the whole set \(V\), we would 
have \mbox{\(n_u\prod_k q_k = 12\)} variables. Considering each channel individually decreases the parameter space of the optimization
and \newcor{therefore} increases the scalability of the algorithm.
It also makes it easier to enforce invertibility of the 
transformation \(\Psi\), which is crucial in control applications.
This is discussed further in Section \ref{sec:invertibility_psi}.
\end{example}
\exampleskip
\newcor{
\paragraph{Representation of quantized inputs} All the input vectors considered further in this section are elements of either  $U$ or $V$,
and can be represented by a linear operator \corplane{$\mathcal{L}: \mathbb{R}^{\sum_k q_k} \rightarrow \mathbb{R}^{n_u}$}.
Let us assume that we have an input $u' \in U$ and a corresponding lifted input $v' \in V$.
We can write $u'$ as
\begin{equation}
    u' = \mathcal{L}'(U_1,U_2,\ldots,U_{n_u}),
\end{equation}
where $\mathcal{L}'$ selects appropriate values from the channels $U_k$
in order to build the vector $u'$. The same operator can be used for the lifted vector $v'$ as
\begin{equation}
    v' = \mathcal{L}'(V_1,V_2,\ldots,V_{n_u}).
\end{equation}
% The action of the linear operator $\mathcal{L}$ is equivalent
% to a multiplication by a matrix $L$.
Action of a linear operator can be represented by matrix multiplication;
we shall represent the operator $\mathcal{L}$ by a matrix $L$.
For $u'$ we get
\begin{equation}
\label{eq:input_projection}
    u' = L' \begin{bmatrix}
        U_1^\top & U_2^\top & \ldots & U_{n_u}^\top,
    \end{bmatrix}^{\top}
\end{equation}
and similarly for $v'$, with the same matrix $L'$. The matrix $L$ (and $L'$) is a block-diagonal matrix 
with row one-hot vectors on the diagonal, selecting only one element from each channel; it shall be referred to as \textit{projection matrix}.

% Note that the matrices $L$ will be unique, as long as the channels $U_k$
% contain unique values.

The Example \ref{ex:matrixL} demonstrates the linear mapping via the projection matrix $L$.

\begin{example}
\label{ex:matrixL}
Assume $u' = \begin{bmatrix}
    -1 & 0.4
\end{bmatrix}^\top$,
$U_1 = \{-1,1\}$, and $U_2 = \{-1,0.4,1\}$.
The equation \eqref{eq:input_projection} would then look as 
\begin{align}
\begin{split}
	u' &= L' \begin{bmatrix}
    U_1 \\ U_2
\end{bmatrix} \\ 
\begin{bmatrix}
    -1 \\ 0.4
\end{bmatrix} &= 
\begin{bmatrix}
    \begin{bmatrix}
        1 & 0
    \end{bmatrix} & \begin{bmatrixcolor}[white]
        0&0&0
    \end{bmatrixcolor} \\ 
    \begin{bmatrixcolor}[white]
        0&0
    \end{bmatrixcolor} & \begin{bmatrix}
        0 & 1 & 0
    \end{bmatrix}
\end{bmatrix}
\begin{bmatrix}
    -1 \\ 1 \\ -1 \\ 0.4 \\ 1
\end{bmatrix},
\end{split}
\end{align}
therefore $L' = \textrm{bdiag}(\begin{bmatrix}
        1 & 0
    \end{bmatrix},\begin{bmatrix}
        0 & 1 & 0
    \end{bmatrix})$.
    
Assume that we have the lifted input channels defined as $V_1 = \{1,0.2\}$, and $V_2 = \{4,3,0.1\}$. We can calculate the lifted vector $v'$ which corresponds to $u'$ by using the same matrix $L'$ as
\begin{equation}
    v' =
    L'
    \begin{bmatrix}
      V_1 \\ V_2  
    \end{bmatrix}
    =
    \begin{bmatrix}
    \begin{bmatrix}
        1 & 0
    \end{bmatrix} & \begin{bmatrixcolor}[white]
        0&0&0
    \end{bmatrixcolor} \\ 
    \begin{bmatrixcolor}[white]
        0&0
    \end{bmatrixcolor} & \begin{bmatrix}
        0 & 1 & 0
    \end{bmatrix}
\end{bmatrix}
    \begin{bmatrix}
        1 \\ 0.2 \\4 \\3 \\ 0.1
    \end{bmatrix}
     = \begin{bmatrix}
         1 \\ 3
     \end{bmatrix}.
\end{equation}
The correspondence between $u'$ and $v'$ is 
therefore explicitly established by sharing the same projection matrix $L'$.
\end{example}
\exampleskip
}
\paragraph{Dataset for learning} 
% \MK{$\mathcal{T}$ is not defined.}
A dataset of \(N\) trajectories of the system \eqref{eq:nlsys} \corr{is denoted as \(\mathcal{D}\)}
% \mbox{\(\mathcal{D} = \{  \mathcal{T}_i : \forall i \in \mathbb{Z}_{1,N} \} \)},
 where 
a single trajectory \(\mathcal{T}_i \in \mathcal{D}\) with initial condition \(x^i_0 \in X_0\) and length \(H_{\rm T}\) is defined as

% Let us define the set of trajectories
% of the system \eqref{eq:nlsys} used for learning the predictor as \(\mathcal{T}\).
% % \MK{Trajectories of what length? The writing is generally a bit vague in places}.
%  A single trajectory \(\mathcal{T}_i \in \mathcal{T}\)  with initial condition \(x^i_0 \in X_0\) and length \(H_{\rm T}\) is defined as 

% \MK{Neni to divny, ze znas $x_0^i$, ale ne $x_s^i$, $s \ge 1$? Klidne muzes predpokladat, ze mas $ ((x^i_{s})_{s=0}^{H_T},(u^i_s)_{s=0}^{H_T-1}, (y^i_{s})_{s=0}^{H_T})$ }
% \newcor{
% \begin{align}
% \begin{split}
% 	\mathcal{T}_i = \{ & (x^i_{s},u^i_s,y^i_{s}) \in X \times U \times Y \\ 
% 	& \text{ s.t. } 
% 	 x^i_{s+1} = f(x^i_s,u^i_s) \quad \forall s \in 0\dots H_{\rm T}-1\\ 
% 	 &\qquad y^i_s = g(x^i_s) \quad \forall s \in 0\dots H_{\rm T}\\
% 	% &\qquad \forall s \in 0\dots H_{\rm T}-1\\
% 	% \text{ for } x_0 = x_0^i \\ 
%    &\qquad \text{and } u^i_s = L^i_s \begin{bmatrix}
%         U_1^\top & \ldots & U_{n_u}^\top
%     \end{bmatrix}^\top 
%     \},
% \end{split}
% \end{align}
% }
	
\begin{equation}
\begin{alignedat}{2}
	\mathcal{T}_i = \{ & (x^i_{s},u^i_s,y^i_{s}) &&\in X \times U \times Y \\ 
	& \text{ s.t. } 
	 x^i_{s+1} &&= f(x^i_s,u^i_s) \quad \forall s \in 0\dots H_{\rm T}-1\\ 
	 &\qquad y^i_s &&= g(x^i_s) \quad \forall s \in 0\dots H_{\rm T}\\
	% &\qquad \forall s \in 0\dots H_{\rm T}-1\\
	% \text{ for } x_0 = x_0^i \\ 
   &\qquad \text{and } u^i_s &&= L^i_s \begin{bmatrix}
        U_1^\top & \ldots & U_{n_u}^\top
    \end{bmatrix}^\top 
    \},
\end{alignedat}
\end{equation}
% \COR{where \(H_{\rm T}\) is the length of the trajectories,}
	\newcor{where $L^i_s$ is a projection matrix 
associated with the input $u^i_s$.
}
% The dataset of \corr{\(N\)} trajectories
%  is denoted as 
%  \mbox{\(\mathcal{D} = \{  \mathcal{T}_i \in \mathcal{T} : \forall i \in \mathbb{Z}_{1,N} \} \)}, where 
% \(N\) is the number of trajectories. 
% \MKres{It seems that $\mathcal{D}$ and $\mathcal{T}$ are the same thing. VC: \(\mathcal{T}\) is supposed 
% to contain all trajectories, \(\mathcal{D}\) contains \(N\) of them.} \MK{What is all trajectories? Does it mean all possible trajectories of the system? If so, then it does not need a separate letter.}
\paragraph{Lifting functions} 
We will obtain the state lifting function \newcor{ \(\Phi: \mathbb{R}^{n_z} \rightarrow \mathbb{R}^{n_x}\)
by pairing the vectors $x_0^i$ and $z^i_0$ as 
\begin{align}
\begin{split}
    \label{eq:phidef}
    z_0^i &= \Phi(x_0^i)\qquad \forall i \in \mathbb{Z}_{1,N},
\end{split}
\end{align}
where $z^i_0 \in Z_0$ and $x^i_0 \in X_0$.
}
In this way, we obtain the function $\Phi$ on \(N\) samples. This function is subsequently extended to all of \(X\) through interpolation; this is detailed in Section \ref{sec:interpolation_phi}.
% by pairing corresponding vectors of the sets \(X_0\) and \(Z_0\),
% meaning that \(\Phi: X_0 \rightarrow Z_0 \). The function then reads
% \begin{align}
% \begin{split}
%     \label{eq:phidef}
%     z &= \Phi(x)\qquad \forall (x,z) \in \mathcal{C}
% \end{split}
% \end{align}

Due to the channel separation, the input lifting function \(\Psi\) is built by concatenating individual functions \mbox{\(\Psi_k : \mathbb{R}\rightarrow \mathbb{R}\)}.
The \(j^\text{th}\) element of the \(k^\text{th}\) channel input vector is transformed as

% \MK{Before the subscript $k$ in $u_k$ denoted time. Now it denotes channel.}
% \COR{VC: SUBSCRIPTS HOPEFULLY FIXED, discrete time is now \(t\), kan\'aly jsou \(k\).}
\begin{align}
	\begin{split}
		v^j_k = \Psi_k(u^j_k) \quad \forall j \in \mathbb{Z}_{1,q_k}, \\\text{ where }
		 u^j_k \in U_k, v^j_k \in V_k.
	\end{split}
\end{align}
\newcor{We} can write \(\Psi: \mathbb{R}^{n_u} \rightarrow \mathbb{R}^{n_u}\) as
\begin{equation}
	\label{eq:psidef}
	% v = \Psi(u) \Leftrightarrow
	\begin{bmatrix}
		v_1 \\ \vdots \\ v_{n_u}
	\end{bmatrix} = 
	\begin{bmatrix}
		\Psi_1(u_1)\\ \vdots \\ \Psi_{n_u}(u_{n_u})
	\end{bmatrix}.
\end{equation}

With the definitions \eqref{eq:phidef} and \eqref{eq:psidef}, optimizing over samples of 
\(Z_0\) and \(V_k\) could be viewed as optimizing over the image of the lifting functions \(\Phi\) and \(\Psi\) \newcor{respectively}.

% \greennote{
%  
% The lifted space of initial conditions \(Z_0\) will be necessarily contained in the actual lifted space \(Z\),
% but as the number of data goes to infinity, we can write 
% \begin{equation}
% 	\lim_{N \rightarrow \infty} Z - Z_0 = \emptyset.
% \end{equation}
%  we'd probably need to also ensure that the trajecotries 
% do not leave the space and that the optimal value of the problem goes to \(\infty\)..

% Note from first review:
% zadefinuju si set \(X\), nasampluju na \(X_0\), porad ale nemam 
% zaruceno ze zustanu alespon v tom \(X_0\) - trajectorie mi muzou utect jinam a pak se vratit
% }

\subsection{Optimization problem setup}
\label{sec:problem_setup}
As a reminder, the output of the linear system \eqref{eq:koop_LTI} 
at time \(k\) for initial 
condition \(z_0^i \in Z_0\) is 
\begin{align}
\begin{split}
	\label{eq:prediction_nice}
	\hat{y}^i_t &= C \left( A^t z^i_0 + \sum_{j=0}^{t-1} A^{t-1-j} B v_j \right) 
	% &\text{  for } z_i = \Phi(x_i), x_i \in X_0, \text{ and } z_i \in Z_0.
\end{split}
\end{align}
% \MK{Why mentioning ``\text{  for } $z_i = \Phi(x_i), x_i \in X_0, \text{ and } z_i \in Z_0.$''?}

To find the Koopman predictor, we will seek to solve the problem
\newcor{
\begin{align}
\begin{split}
	\label{eq:problem_description}
\minimize & \sum_{i=0}^{N} \sum_{t=0}^{H_{\rm T}}
||  C z_t^i - y^i_t ||_2^2 + w_0\sum_{(a,b)\in \mathcal{S}} ||z^a_{H_{\rm T}} - z^{b}_{0}||_2^2 + \theta(\cdotp) \\ 
  \text{s.t. }& z^i_t = \left( A^t z^i_0 + \sum_{j=0}^{t-1} A^{t-1-j} B v^i_j \right) \\ 
  % &  \hat{y}_t^i = C z_t^i \\
 % & v_p \in V_1 \times V_2 \times \dots V_{n_u}, z_0^k \in Z_0, y_t^i \in \mathcal{D},
 & v^i_j = L^i_j \begin{bmatrix}
     V_1^\top & \ldots & V_{n_u}^\top
 \end{bmatrix}^\top
 % & z_0^k \in Z_0, y_t^i \in \mathcal{D},
\end{split}
\end{align}
where \corplane{\(A,B,C,z_0^i,V_k\)} are the decision variables,
\corplane{\( \mathcal{S}, L^i_j,y^i_t\in\mathcal{D}\)} are the problem data,}
and \(H_{\rm T}\) and \(N\) are the trajectory length and number of trajectories respectively. \newcor{The scalar $w_0$ is a weighting
parameter.}
The set \(\mathcal{S}\) contains indices of consecutive trajectories so for any \((a,b)\in \mathcal{S}\), \(\mathcal{T}_a\) ends at the same point where \(\mathcal{T}_b\) begins; \newcor{the reasoning behind this 
regularization is explained further below, in the Section \ref{sec:trajectory_preparation}}. 
% \MK{In (\eqref{eq:problem_description}) you need to specify what $Z_0$ is. To me it seems that it is a user specified subset of $\R^{n_z}$. Given the data-driven nature it seems that it can be chosen without loss of generality to be $[-1,1]^{n_z}$. Same for the input sets. Without loss of generality, you can say that $v_{j}^i \in [-1,1]$, no? I don't see the need for the discrete sets $V_1,\ldots,V_{n_u}$ Also, what is $v_p$? }
% \newcor{VC: The sets can be initialized like this, but it is difficult to constrain them because I am using unconstrained solver. 
% I added a initialization reference to the section \ref{sec:basic_def},
% pls change/delete your(and mine) comment yourself if it's resolved.}

The first term in the cost function ensures the fit 
of the data; the second one promotes invariance of \(Z_0\) by connecting the consecutive trajectories in the lifted space, it can also be thought of as a regularization term.
The last term, \(\theta\), is a placeholder for optional regularization, such as enforcing invertibility
and symmetry of the input transformation \(\Psi\)
 which is discussed in the Section \ref{sec:implementation}
 and below Corrolary \ref{sec:corrolary} respectively.
\vclast{The problem \eqref{eq:problem_description} can be simplified in terms of the number of variables and 
data requirements by exploiting symmetries of system \eqref{eq:nlsys}, which will be discussed in the next section.}

\section{Exploiting symmetry}
% \textbf{UDELEJ Z TOHODLE SEKCI a nemej to v implementation details.}
\label{sec:symmetry}

Assuming that the nonlinear system \eqref{eq:nlsys} has a symmetry, we can exploit 
this symmetry in our algorithm to decrease the size of the dataset by imposing structure onto the matrices \((A,B,C)\), which will guarantee that the Koopman predictor will respect the same symmetries as \eqref{eq:nlsys}. This means that the learning dataset \(\mathcal{D}\) will not need to contain symmetric trajectories, since the symmetry will be implicitly enforced by the structure of \((A,B,C)\), thus decreasing the size of \(\mathcal{D}\), $Z_0$, and the number of decision variables in \((A,B,C)\).

We consider state-control symmetries with respect to groups $\Gamma^x \subset \mathrm{GL}_{n_x}$ and $\Gamma^u \subset \mathrm{GL}_{n_u}$, where $\mathrm{GL}_n$ denotes the group of invertible matrices of size $n$-by-$n$. The group elements are denoted by $\gamma^x \in \Gamma^x$ and $\gamma^u\in\Gamma^u$ and the group action is the standard matrix multiplication. The two groups are assumed to be related by a group homomorphism $h:\Gamma^x\to\Gamma^u$, i.e., $h(\gamma^x_1\gamma^x_2) = h(\gamma_1^x)h(\gamma_2^x)$. A dynamical system is said to have a symmetry with respect to $(\Gamma^x,\Gamma^u,h)$ if for all $\gamma^x \in \Gamma^x$ it holds
\begin{equation}
	\label{eq:symmetry_def}
	f(\gamma^x x, \gamma^u u) = \gamma^x f(x,u),
\end{equation}
where $\gamma^u = h(\gamma^x)$.
Our goal is to find a linear predictor whose output will respect the symmetry with respect to $\Gamma^x$. In order to do so, we will construct a symmetry group $\Gamma^z\subset \mathrm{GL}_{n_z}$ and a group homomorphism $h':\Gamma^{x}\to \Gamma^z$ such that
\begin{equation}
\begin{alignedat}{2}
	\label{eq:symmetry_lti_def}
	A \gamma^z z + B \gamma^u v &= \gamma^z (Az + Bv) \\ 
	C \gamma^z z &= \gamma^x Cz 
\end{alignedat}
\end{equation}
for all $\gamma^x\in \Gamma^x$, where $\gamma^z = h'(\gamma^x)$ and $\gamma^u = h(\gamma^x)$. This implicitly assumes that the input symmetries are  the same for the original system and the predictor;
the requirements on the transformations  \(\Phi\) and \(\Psi\) for this to hold are stipulated later in Corollary \ref{sec:corrolary}.

To simplify the exposure, this section assumes  
that the output of the predictor \eqref{eq:symmetry_lti_def}
is prediction of the nonlinear state \(x\).
Furthermore, we shall consider only sign symmetries, meaning that $(\Gamma^x,\Gamma^u,\Gamma^z)$ are subsets of diagonal matrices with $+1$ or $-1$ on the diagonal.

% \corr{The elements of \(\Gamma^x\)
% will be parametrized by a vector \(s = [s_{1},\dots, s_{n_\Gamma}]^\top \in \{-1,1\}^{n_{\Gamma}}\), where \(n_{\Gamma} \leq n_{x}\) denotes the 
% number of symmetry axes.
% In the case where some states do not have any symmetry,
% we assign them \(s_1 = 1\) and the number of symmetry 
% axes would them be \(n_{\Gamma} - 1\).

% We use the operator \(\mathcal{I}: \{-1,1\}^{n_\Gamma} \rightarrow \{-1,1\}^{n_x}\) to map the sign changes \(s\) onto the individual states, such that \(\mathcal{I}(s)_{i} \in s \quad \forall i = 1 \dots n_x\).

% The vector of sign symmetries \(s\) and the mapping 
% of sign changes onto the state vector \(\mathcal{I}\)
% determine the elements of \(\Gamma^x\) for a concrete system,
% such that

% \begin{equation}
% 	\Gamma^x = \{
% 	\text{Diag}(\mathcal{I}(s)) : s_i \in \{-1,1\}
% 	\}	
% \end{equation}

	% For a concrete system \(f\), the set diagonal 
	% matrices \(\Gamma^x\) is 
	% \begin{equation}
	% 	\Gamma^x = \{
	% 		\gamma^x : f(\gamma^x x ,\gamma^u u) = \gamma^x f(x,u), \gamma^u = h(\gamma^x)
	% 	\}
	% \end{equation}
	We shall simplify the notation by denoting the diagonal
	matrices	as vectors containing the diagonal
	whenever clear from the context.

\begin{example}
	\label{ex:symm_gammas}
Assume the following discrete dynamical system
\(x^+ = f(x,u)\) (we drop the index \(k\) for readability):
\begin{equation}
	\label{eq:symm_example}
\begin{alignedat}{2}
x^+_1 &= -x_{1}u_1 - |x_{2}|\\
x^+_{2} &= -x_{2} + u_2 + u_3\\ 
x^+_{3} &= -x_{3}|x_{2}| + u_2 + u_3\\
x^+_{4} &= -x_{4}.
\end{alignedat}
\end{equation}
The set \(\Gamma^x\) is 
\begin{equation}
	\label{eq:ex_Gammax}
	\Gamma^x = \left\{ 
		\begin{bmatrix}
		1\\1\\1\\1
	\end{bmatrix},
	\begin{bmatrix}
		1\\-1\\-1\\1
	\end{bmatrix},
	\begin{bmatrix}
		1\\1\\1\\-1
	\end{bmatrix},
	\begin{bmatrix}
		1\\-1\\-1\\-1
	\end{bmatrix} 
	\right\}
\end{equation}
%\begin{equation}
	%h(p) = \begin{bmatrix}
%		p_1 & p_2 & p_2
	  %\end{bmatrix}.
%\end{equation}
and the pairs \((\gamma^x,\gamma^u)\) can take following values:
\begin{equation}
\begin{alignedat}{2}
	&(\gamma^x,\gamma^u) = (\gamma^x,h(\gamma^x)) \in \\
	&\left\{
	\left(
	\begin{bmatrix}
		1 \\ 1 \\ 1\\1
	\end{bmatrix},
	\begin{bmatrix}
		1\\1\\1
	\end{bmatrix}
	\right),
	\left(
		\begin{bmatrix}
			1 \\ -1 \\ -1\\1
		\end{bmatrix},
		\begin{bmatrix}
			1\\-1\\-1
		\end{bmatrix}
		\right),
		\left(
			\begin{bmatrix}
				1 \\ 1 \\ 1\\-1
			\end{bmatrix},
			\begin{bmatrix}
				1\\1\\1
			\end{bmatrix}
			\right),
			\left(
				\begin{bmatrix}
					1 \\ -1 \\ -1\\-1
				\end{bmatrix},
				\begin{bmatrix}
					1\\-1\\-1
				\end{bmatrix}
				\right)\right\}.
\end{alignedat}
\end{equation}
\iffalse
We can parametrize \(\gamma^x\) by a vector \(s \in \{-1,1\}^{3}\) and we fix  \(s_1 = 1\) to denote the first state 
that does not have any symmetry.
Therefore 
\(\gamma^x = \begin{bmatrix}
	1 & s_2 & s_2 & s_3
\end{bmatrix}\) and \(\gamma^u = h(\gamma^x) = \begin{bmatrix}
	1 & s_2 & s_2
\end{bmatrix}\).

We verify these actions by plugging them into \eqref{eq:symmetry_def}
\begin{equation}
\begin{alignedat}{2}
	f(\gamma^x x, \gamma^u u) &=
	\begin{bmatrix}
		-1^2 x_1 u_1 - |s_2 x_2| \\ 
		-s_2 x_2 + s_2 u_2 + s_2 u_3\\
		-s_2 x_3 |s_2 x_2| + s_2 u_2 + s_2 u_3\\
		-s_3 x_4
	\end{bmatrix}\\ &= 
	\begin{bmatrix}
		1 ( -x_{1}u_1 - |x_{2}|)\\
		s_2 ( -x_{2} + u_2 + u_3)\\
		s_2 (-x_{3}|x_{2}| + u_2 + u_3)\\
		s_3 ( -x_{4})
	\end{bmatrix} = \gamma^x f(x,u).
\end{alignedat}
\end{equation}
\fi
% \corr{The matrix \(H\) is 
% \begin{equation}
% 	H = \begin{bmatrix}
% 		1 & 0 & 0\\ 0 & 1 & 1\\ 0 & 1 & 1 \\ 0&0&0
% 	\end{bmatrix}.
% \end{equation}}
% which can be verified by direct application of
% \eqref{eq:symmetry_xu_realtion}
% \begin{equation}
% 	\label{eq:H_real_entries}
% 	H\gamma^u  + \gamma^x_{\rm A} =\begin{bmatrix}
% 		1 & 0 & 0\\ 0 & \nicefrac{1}{2} & \nicefrac{1}{2}\\ 0 & \nicefrac{1}{2} & \nicefrac{1}{2} \\ 0&0&0
% 	\end{bmatrix} \begin{bmatrix}
% 		1 \\ s_2 \\ s_2
% 	\end{bmatrix}  + \begin{bmatrix}
% 		0\\0\\0\\s_3
% 	\end{bmatrix} = 
% 	\begin{bmatrix}
% 		1 \\ s_2 \\ s_2 \\s_3
% 	\end{bmatrix} = \gamma^x.
% \end{equation}
\end{example}
\exampleskip
We see that some states can 
change their signs together (\(x_2\) and \(x_3\)
in the Example \ref{ex:symm_gammas}) and it will be useful to 
group these together. To this end, let \[
\mathcal{I}({\Gamma^x}) = (\mathbb{I}_1,\ldots,\mathbb{I}_{n_{\Gamma}})
\]
where each index set $\mathbb{I}_i \subset \mathbb{Z}_{1,n_x}$ satisfies the following two conditions:
\begin{enumerate}
     \item $\gamma_j^x = \gamma_k^x$ for all $j,k\in \mathbb{I}_i$ (states indexed by $\mathbb{I}_i$ change sign together).
     \item $\mathbb{I}_i$ is maximal (no indices can be added to $\mathbb{I}_i$  without violating the first condition)
\end{enumerate}
This implies that the index sets $\mathbb{I}_i$ are disjoint and hence $\sum_{i=1}^{n_{\Gamma}} |\mathbb{I}_{i}| = n_x$. We shall assume that the index sets are ordered in an increasing order, i.e., if $i > j$, then $k > l$ for all $k \in \mathbb{I}_i$ and $l \in \mathbb{I}_j$. This can be achieved without loss of generality by reordering the states. Coming back to Example~\ref{ex:symm_gammas},
	we get
\(\mathcal{I}({\Gamma^x}) = (\{1\},\{2,3\},\{4\} )\).

In order to enforce the symmetry in the Koopman predictor, 
we will define  the groups \(\Gamma^z\) and \(\Gamma^v\) for the lifted vectors \(z \in Z\) and \(v \in V\). 
% Since \(v\) has the same dimensions as \(u\),
% and the transformation between them is done channel-wise, the action \(\gamma^v\) is equal to \(\gamma^u\).
For the input, we use the assumption \(\Gamma^v = \Gamma^u\).
For defining \(\Gamma^z\), we need to impose some structure
onto the vector \(z\).
 We do this by fixing the sparsity pattern of the \(C\) matrix in order to explicitly link the elements of \(z\) with 
 the elements of \(x\). 
 The structure of the matrix \(C\)
 will be block-diagonal
 \begin{equation}
	\label{eq:C_symm_def}
	C = \text{bdiag}(c_1^\top,\ldots,c_{n_x}^\top),
 \end{equation}
 where 
 $c_i$ are vectors of user-specified length; the lengths determine
 how many elements of \(z\) will be used for reconstruction of 
 elements of \(x\) which can be retrieved as 
%  \corr{, in this work we use \(|c_i| = \nicefrac{n_z}{n_x}\quad \forall i = 1\dots n_x\).}
%  \MK{How do you specify the length?}.
 % \MK{Write down that $C = \mathrm{bdiag(c_1,\ldots,c_{n_x})}$. Just saying it is block diagonal is not enough.}
%   The \(i^{\text{th}}\) element of \(x\) can be retrieved as 
\begin{equation}
	\label{eq:C_def}
	x^i = c_i^\top z^{c_i},
\end{equation}
where \(z^{c_i}\) is a part of the vector \(z\) which corresponds to $c_i$ and
has length $|c_i|$. In this work we use \(|c_i| = \nicefrac{n_z}{n_x},\; i = 1\dots n_x\).
The whole vector $z$ can be written as
$z = \begin{bmatrix}
    {z^{c_1}}^\top & \ldots & {z^{c_{n_x}}}^\top
\end{bmatrix}^\top$.
% Then we get  \MK{Where does $\gamma^x x = \gamma^x Cz $ follow from? The predictions are not exact.}
% \begin{equation}
% \begin{alignedat}{2}
% \gamma^x x = \gamma^x Cz = \gamma^x \begin{bmatrix}
% 	c_1^\top z^{c_1} &\dots & c_i^\top z^{c_i} &...&  c_{n_x}^\top z^{c_{n_x}}
% \end{bmatrix}^\top = \\
% \begin{bmatrix}
% 	c_1^\top \gamma^x_{1,1}z^{c_1} &\dots & c_i^\top \gamma^x_{i,i}z^{c_i} &...&  c_{n_x}^\top \gamma^x_{n_x,n_x}z^{c_{n_x}}
% \end{bmatrix}^\top = C \gamma^z z .
% \end{alignedat}
% \end{equation}
% Therefore
% \MK{Why the double indexing of $\gamma^x$. It is a vector, not a matrix, regardless of the overloading of $\gamma^x M = \mathrm{diag}(\gamma^x)$ M.}
The set \(\Gamma^z\) is defined as
\begin{equation}
	  \Gamma^z =  \{ \text{bdiag}
(		\gamma^x_1 I_{|c_1|}, \dots, \gamma^x_{n_x} I_{|c_{n_x}|})
: \gamma^x \in \Gamma^x \},
\end{equation}
so that the elements of \(\gamma^x\) are multiplied with the 
corresponding parts of \(z\), i.e.,
\begin{equation}
	\label{eq:gammaz_def}
	\gamma^z z = \begin{bmatrix}
		\gamma^x_{1}I_{|c_1|} & & \\ 
		& \ddots & \\ 
		&& \gamma^x_{n_x}I_{|c_{n_x}|}
	\end{bmatrix}\begin{bmatrix}
		z^{c_{1}} \\ \vdots \\ z^{c_{n_x}}
	\end{bmatrix} = \begin{bmatrix} \gamma_1^x z^{c_1} \\ \vdots \\ 
 \gamma^x_{n_x} z^{c_{n_x}}  \end{bmatrix}.
\end{equation}
% \MK{Can you define $\gamma^z$ directly using (\eqref{eq:gammaz_def}) without Eq. (32)? }
% \newcor{
% Finally, we can write 
% \begin{equation}
% \label{eq:gamma_z_def}
%     \gamma^z_k = \gamma^x_l \iff C_{l,k} \neq 0.
% \end{equation}
Taking into consideration that 
some states change signs together (\(|\mathbb{I}_i| > 1\) for some \(i\)), we can write the same set as
% Taking into consideration the (possibly) repeating signs \MK{What do you mean by repeating signs?} we can write the same set as 
% \MK{It seems that these two sets are equal only after reordering of the $z$ states. If this is the case, specify it.}
% \begin{equation}
% 	\label{eq:gamma_z_blockA}
% 	\textrm{Diag}(\gamma^z(s)) = \textrm{bdiag}(s_{1}I_{n_1},\dots,
% 	s_{n_\Gamma}I_{n_{\Gamma}}) \quad \forall s \in \{-1,1\}^{n_\Gamma}, 
% 	\text{where } 
% 	n_i = \sum_{k \in \mathcal{I}({\Gamma^x})_i}|c_k|
% \end{equation}
% \MK{Is $\nu_i = |\mathbb{I}_i|$? VC: No, \(\nu_i\) 
% depends also on the user-selected \(c_i\)'s }
\begin{equation} 
	\label{eq:gamma_z_blockA}
	\Gamma^z  = \left\{ \text{bdiag}
(		\gamma^x_{\mathbb{I}_1} I_{\nu_1}, \dots, \gamma^x_{\mathbb{I}_{n_{\Gamma}}} I_{\nu_{n_\Gamma}})
: \gamma^x \in \Gamma^x, \mathbb{I} \in \mathcal{I}(\Gamma^x), 
\nu_i = \sum_{k \in \mathbb{I}_i}|c_k|
 \right\},
\end{equation}
% \MK{Should be $I_{\nu_{n_\Gamma}}$? VC: yes}
where \(\gamma^x_{\mathbb{I}_i}\) is to be understood as scalar, since 
all the elements of \(\gamma^x\) have the same value
 at indices \(\mathbb{I}_i\) by definition. The scalars \(\nu_i\) represent the number of lifted states corresponding to the original states indexed by \(\mathbb{I}_i\), therefore we get \(\sum_{i=1}^{n_\Gamma} \nu_i = n_z\).
 Notice that \(\Gamma^z\), and therefore \(h'\), depends on the pattern of \(C\) which 
 is the connecting element between the vectors \(x\) and \(z\).
 This means that not only the dimension, but also the structure of the lifted space can be tuned in order 
 to obtain good prediction performance.

The matrix $A$ will have a block-diagonal sparsity pattern
with the same block sizes as in \eqref{eq:gamma_z_blockA}.
% with dense matrices on the diagonal 
% \MK{Nerozumim proc je pocet bloku v $A$ roven $n_x$. }
\begin{equation}
	\label{eq:A_symmetry_def}
    A = \textrm{bdiag}(A_{1},\ldots,A_{n_{\Gamma}}),\text{ where }
    A_{i} \in \mathbb{R}^{\nu_i \times \nu_i} \text{ and }
	\nu_i = \sum_{k \in \mathcal{I}(\Gamma^x)_i}|c_k|.
\end{equation}
% The matrix $B$ will have a sparsity pattern similar to the matrix $H$, except in vectorized form:
% \MK{Seems like this is the only place where you are using $H$. I propose to erase $H$ and define $B$ directly.}
The matrix \(B \in \mathbb{R}^{n_z \times n_u}\) will also have block-diagonal sparsity 
pattern:
\begin{equation}
\label{eq:B_symmetry_def}
    % B = \begin{bmatrix}
    %     B^1 \\ \vdots \\ B^{n_x}
    % \end{bmatrix},\;\;
    % % \textrm{ for } 
	B_{i,k} \in \begin{cases}
        \mathbb{R} &\text{ if } h'(\gamma^x)_i = h(\gamma_x)_k\,\forall\gamma^x\in\Gamma^x\\ 
        0 & \text{ otherwise }
    \end{cases}
	%\text{ for }
	%\gamma^z = h'(\gamma^x),
	%\gamma^v = h(\gamma^x),
	%\forall \gamma^x \in \Gamma^x,
    % B^i_{:,j} \in \begin{cases}
    %     \{0\}^{|c_i|} &\text{ if } H_{i,j} = 0\\ 
    %     \mathbb{R}^{|c_i|} & \text{ otherwise \MK{Vysvetli\;} } \MK{\mathbb{R}^{|c_i|}}
    % \end{cases},
\end{equation}
where $B_{i,k} = \mathbb{R}$ signifies that the entry is not fixed to zero. 
\corr{
Coming back to the Example \ref{ex:symm_gammas} ,
the predictor would have the structure 
\begin{equation}
\begin{alignedat}{2}
z^+ &=  \begin{bmatrix}
    A_{11} & 0      & 0       & 0 \\
    0      & A_{22} & A_{23}  & 0 \\
    0      & A_{32} & A_{33}  & 0 \\
    0      & 0      & 0       & A_{44}
\end{bmatrix}z + \begin{bmatrix}
    B_{11} & 0      & 0 \\
    0      & B_{22} & B_{23} \\
    0      & B_{32} & B_{33} \\
    0      & 0      & 0 \\
\end{bmatrix}v \\
\hat{x} &= \begin{bmatrix}
	c_1^\top & 0 & 0 &0\\
    0 & c_2^\top & 0 &0\\
    0 & 0 & c_3^\top &0\\
	0 & 0 & 0 & c_4^\top\\
\end{bmatrix}z,
\end{alignedat}
\end{equation} 
where \(c_i\) are vectors with user-selected lengths, and 
\(A_{i,j}\) and \(B_{i,j}\) are matrices of appropriate sizes.
}

We say that the LTI system \(z^+ = Hz + Gv, x = Fz\) respects the output and input  symmetries \(\gamma^x\) and \(\gamma^v\) respectively, if for every trajectory of the LTI system \((z^a_t,x^a_t,u^a_t)_{t = 0}^{\infty}\),
there exists a sequence $(z_t^b)_{t=0}^{\infty}$
such that \((z^b_t,\gamma^x x^a_t,\gamma^u u^a_t)_{t=0}^{\infty}\) is also its trajectory.

\begin{lemma}
	\label{lemma:symmetric_lti}
	% {\color{red}MK ... then there exists $z_t^b$ such that (z,x,u) is a trajectory of the system}
	Assume an observable LTI system \(z^+ = Hz + Gv, x = Fz\), where the output \(x\) 
	and input \(u\) respect the symmetries \(\Gamma^x\) and \(\Gamma^v\).
	% , 
	%  such that if \((z^a_t,x^a_t,u^a_t)_{t = 0}^{\infty}\) is a trajectory of the LTI system,
	%  then there exists a sequence $(z_t^b)_{t=0}^{\infty}$
	%  such that \((z^b_t,\gamma^x x^a_t,\gamma^u u^a_t)_{t=0}^{\infty}\) is also a trajectory.
	% such that both \((x_t,u_t)\) and \((\gamma^x x_t,\gamma^u u_t)\) are feasible trajectories of the system for all \(t\) and some initial conditions \(z^1\) and \(z^2\).
	The system can be transformed into an equivalent form,
	which has the same output \(x\), its state respects the symmetry \(\Gamma^z\),
	 and its state matrices have 
	the block-diagonal form introduced above.
\end{lemma}
% \begin{lemma}
% 	Assume observable Koopman system \(z^+ = Az + Bv, x = Cz\)
% 	with lifting functions {\(z = \Phi(x)\), \( v = \Psi(u)\)} defined according to \eqref{eq:phidef} and \eqref{eq:psidef} respectively.
% 	If the original nonlinear system \(x^+=f(x,u)\) respects the actions \(\gamma^x\) and \(\gamma^u\),
% 	the Koopman system can be written with the presented sparsity patterns for \(A,B,C\) and it will respect the actions \(\gamma^z\) and \(\gamma^v\).
% \end{lemma}
\begin{proof}
	% \MK{uvod pryc}
	% We start the proof by an observation that any LTI system can be transformed
	% into a form with block-diagonal output matrix $\hat{C} = \text{bdiag}(c_1^\top,\ldots,c_{n_x}^\top)$.
	% % This form will then allow us to use the definitions of \(\gamma^z\) and \(\gamma^v\),
	% % which must be respected 
	% % If the output of the system has the symmetry \(\gamma^x\)
	% % and the matrix \(C\) is in the block-diagonal form, the state must have the symmetry \(\gamma^z\). This
	%  This imposes the symmetry \(\gamma^z\) onto the state.
	% We conclude by showing that an LTI system with symmetries \(\gamma^z\) and \(\gamma^v\)
	% has the block-diagonal form introduced earlier.

	% Since we assume that the output of our Koopman system is the nonlinear 
	% state \(x\), the Koopman system must respect the same symmetries as \(f(x,u)\). With block-diagonal \(\hat{C}\), this requirement is transformed
	% into respecting the symmetry actions \(\gamma^z\) and \(\gamma^v\) 
	% in the lifted space.

	% Lastly, we show that by imposing the symmetries \(\gamma^z\) and \(\gamma^v\), the matrices \(A\) and \(B\) must have the 
	% the block-diagonal structure introduced above. \\\\
	% Any linear system can be transformed into an equivalent system with nonsingular matrix \(T\) as 
	
% such that the input and output variables remain unchanged.

	We can assume that the original system
	\((H,G,F)\) has been transformed into its observer form
	\((\hat{A},\hat{B},\hat{C})\)
	\cite[Section 6.4.2]{linear_primer} and its states have been reordered
	such that
	\begin{equation}
		\hat{C} = \begin{bmatrix}
			1 & 0 & \dots &0 & 0 & \dots &0 \\
			\times & 1 & \dots &0& 0 & \dots &0 \\
 			  &   & \ddots & & 0 & \dots &0\\
			\times & \times & \dots &1& 0 & \dots &0 \\
		\end{bmatrix},
	\end{equation}
	where \(\times\) are unfixed entries.
	We want to find a full-rank matrix \(T\) to perform the similarity transformation
	\begin{subequations} \label{eq:lti_trans}
		% \begin{equation}
		\begin{align}
		{A} &= T^{-1} \hat{A} T \\ 
		{B} &= T^{-1} \hat{B} \\ 
		{C} & = \hat{C}T, \label{eq:lti_trans_output}
		\end{align}
		% \end{equation}
	\end{subequations}
	 with \({C}\) having the block-diagonal structure used in \eqref{eq:C_symm_def}. 
% Let us first define the cumulative sum of sizes of the blocks \(|c_i|\) as
% \begin{equation}
% 	\nu_l = \sum_{i = 1}^{l} |c_i|.
% \end{equation}
We can write \({C} = \hat{C}T\) as
	\begin{equation}
		\label{eq:trans_matrix}
		\begin{bmatrix}
			c_1^\top &0&0&0\\
			0&c_2^\top&0&0\\
			0 & 0 & \ddots & 0\\
			0&0 &0&c_{n_x}^\top
		\end{bmatrix}
		=
		\begin{bmatrix}
			1 & 0&0&0&0 & \dots &0 \\
			\hat{c}_{2,1} & 1&0&0&0&\dots &0\\
			\vdots && \ddots & &0&\dots &0\\ 
			\multicolumn{3}{c}{\hat{c}_{n_x,1:(n_x-1)}} &1&0&\dots &0
		\end{bmatrix}
		\begin{bmatrix*}
			c_1^\top                &0                       &0 &0\\
			  T_{2,1}                        &c_2^\top &0 &0\\
			  \vdots&&\ddots&0\\
			%   \multicolumn{3}{c}{[\hfill T_{n_x,1} \hfill]}&c_{n_x}^\top\\
			T_{n_x,1} & T_{n_x,2} & ... & c_{n_x}^\top \\ 
			\cmidrule(lr){1-4}
			\multicolumn{4}{c}{T_{\rm F}}
		\end{bmatrix*}
	\end{equation}
where \(T_{i,j}\) is a row block of length \(|c_i|\) such that
\begin{equation}
	T_{i,j} =
	\begin{cases}
		 - \sum_{k=1}^{i-1} \hat{c}_{i,k} T_{k,j}, & \mbox{if } i >j\\
		 c_i^\top, & \mbox{if } i = j \\
		 0 & \mbox{if } i < j.
	\end{cases}
\end{equation}
The first \(n_x\) rows of \(T\) will have full row rank with this construction, since the observability assumption 
implies \(c_i \neq \mathbf{0}_{|c_i|}\).
The remainder (matrix \(T_{\rm F}\)) is multiplied by \(0\) in the matrix \(\hat{C}\),
therefore it can be freely chosen such that \(T\) has full rank.

We shall demonstrate the construction of \(T\) on a simple example with
\(n_y=3,n_z = 6\), and \(c_1 = [1,2], c_2 = [3,0], c_3 = [5,6]\). 
The equation \eqref{eq:trans_matrix} is then
	\begin{equation}
		\begin{bmatrix}
			1 & 2&0&0&0&0\\
			0&0&3&0&0&0\\
			0 & 0&0&0&5&6
		\end{bmatrix}
		=
		\begin{bmatrix}
			1 & 0&0&0&0&0\\
			a & 1&0&0&0&0\\
			b & c&1&0&0&0
		\end{bmatrix}\begin{bmatrix}
			1 & 2&0&0&0&0\\
			-a & -2a&3&0&0&0\\
			ac-b & 2(ac-b)&-3c&0&5&6\\
			\cmidrule(lr){1-6}
			\multicolumn{6}{c}{T_{\rm F}}
		\end{bmatrix}.
	\end{equation}
We see that the matrix \(T\) has full rank (provided \(T_{\rm F}\) has full row rank)
for all values of the unfixed entries.

Having the matrix \(C\) in the block-diagonal form,
we will use the definition of \(\Gamma^z\) \eqref{eq:gamma_z_blockA} and the sparsity patterns 
of \(A\) \eqref{eq:A_symmetry_def} and \(B\) \eqref{eq:B_symmetry_def} to show that 
% \MK{vezmu \(\gamma^z,A,B\) a ukazu ze plati (41)}
% The block-diagonal matrix \(C\) and the symmetry \(\gamma^x\) allows us 
% to define the state symmetry \(\gamma^z\) according to \eqref{eq:gammaz_def}. The action \(\gamma^v\) remains the same 
% since the similarity transformation \eqref{eq:lti_trans} does not change the input.

% Having our output matrix in the block-diagonal form,
% we can link the elements of \(x\) and \(z\) via \eqref{eq:C_def} and 
% therefore apply the actions \(\gamma^z\) and \(\gamma^v\) on the Koopman system. 
% This will enforce the original symmetries \(\gamma^x\) and \(\gamma^u\)
% to be maintained through the transformations to and from the lifted space.
% -------------TODO ty lematka jsou trochu divny tady----------------
% Having the symmetry actions 
% defined, we will find matrices \(A,B,C\) for which it holds:
\begin{equation}
	A(\gamma^z z) + B(\gamma^v v) = \gamma^z (Az + Bv)
	\quad \forall \  \Gamma^x,
\end{equation}
where \(\gamma^x \in \Gamma^x\), \(\gamma^v = h(\gamma^x)\),
and \(\gamma^z = h'(\gamma^x)\).
The equation can be separated into two conditions
\begin{equation}
	\label{eq:proof_comm_A}
	A \gamma^z = \gamma^z A,
\end{equation}
and 
% \MK{Not sure I get the explanation here. The matrix $\gamma^z$ has $n_x$ blocks whereas $A$ has $n_\Gamma$ blocks. This relates to my previous comment in red.}
\begin{equation}
	\label{eq:proof_comm_B}
	B \gamma^v = \gamma^z B.
\end{equation}
\corr{
The equation \eqref{eq:proof_comm_A} is a product 
of two block-diagonal matrices with the same block structure
 with 
the definitions \eqref{eq:gamma_z_blockA} and \eqref{eq:A_symmetry_def}. Block-diagonal matrices commute if and only if their blocks commute and since the blocks of \(\gamma^z\) are scalar
matrices, they commute with every matrix and
therefore \eqref{eq:proof_comm_A} holds.
}

% The first condition \eqref{eq:proof_comm_A} holds from Lemma \ref{lemma:commuting_matrices} since both matrices have the 
% same block-structure and
% the blocks 
% of \(\gamma^z\) are scalar matrices which commute with every matrix.
% and the definitions of \(A\) \eqref{eq:A_symmetry_def} and \(\gamma^z(s)\) \eqref{eq:gamma_z_blockA}, since the blocks 
% of \(\gamma^z(s)\) are scalar matrices which commute with every matrix. 
The second equation \eqref{eq:proof_comm_B} can be written as 
% We shall represent the actions \(\gamma^z\) and 
% \(\gamma^v\) by diagonal matrices \(G\) and \(M\) respectively.
% The above equation then reads 
% \begin{equation}
% \begin{alignedat}{2}
% 	AGz + BMv = GAz + GBv \\ 
% 	(AG - GA)z + (BM - GB)v = 0,
% \end{alignedat}
% \end{equation}
% giving us two conditions \(AG - GA = 0\) and \(BM - GB = 0\).
% The first condition, $AG - GA = 0$,
% The second condition is \(BM - GB = 0\).
% Since both matrices \(M\) and \(G\) are diagonal, they are simply 
% scaling the columns and rows of \(B\), respectively.
% We can rewrite the equation by using element-wise products \(\odot\) as 
\newcommand*{\vertbar}{\rule[-1ex]{0.5pt}{2.5ex}}
\newcommand*{\horzbar}{\rule[.5ex]{2.5ex}{0.5pt}}
\begin{equation}
	\begin{bmatrix}
		\vertbar & \vertbar &  & \vertbar \\ 
		\gamma_1^v & \gamma_2^v & \dots & \gamma_{n_v}^v \\ 
		\vertbar & \vertbar &  & \vertbar \\ 
	\end{bmatrix}
	\odot  B - 
	\begin{bmatrix}
		\horzbar & \gamma_1^z & \horzbar \\ 
		\horzbar & \gamma_2^z & \horzbar \\ 
		 & \vdots & \\
		\horzbar & \gamma_{n_z}^z & \horzbar \\ 
	\end{bmatrix}
	\odot B  = 0,
\end{equation}
after factoring out \(B\), we obtain
\begin{equation}
	\label{eq:B_factor}
	\left(
	\begin{bmatrix}
		\vertbar & \vertbar &  & \vertbar \\ 
		\gamma_1^v & \gamma_2^v & \dots & \gamma_{n_v}^v \\ 
		\vertbar & \vertbar &  & \vertbar \\ 
	\end{bmatrix}
	-
	\begin{bmatrix}
		\horzbar & \gamma_1^z & \horzbar \\ 
		\horzbar & \gamma_2^z & \horzbar \\ 
		 & \vdots & \\
		\horzbar & \gamma_{n_z}^z & \horzbar \\ 
	\end{bmatrix}
	\right) \odot B = 0.
\end{equation}
Let us call the matrix in parentheses \(D\) and obtain 
\begin{equation}
	\label{eq:B_condition}
	D \odot B = 0.
\end{equation}
The matrix \(D\) will be 0 only in places, where 
the elements of \(\gamma^v\) and \(\gamma^z\) are equal
\begin{equation}
	\gamma^z_k = \gamma_p^v \implies D_{k,p} = 0\text{ such that } \gamma^v = h(\gamma^x),\gamma^z = h'(\gamma^x),
	\forall \gamma^x \in \Gamma^x.
\end{equation}
It is on these indices where the matrix \(B\) can have nonzero entries,
it has to be zero everywhere else in order for \eqref{eq:B_condition} to hold.
We see from \eqref{eq:B_symmetry_def} that the matrix \(B\) fulfills this condition by definition.

% It is on these indices, where matrix \(B\) is permitted nonzero values
% in order for the equation \eqref{eq:B_factor} to still hold. 
% This will naturally give rise to the block-diagonal structure introduced in \eqref{eq:B_symmetry_def}.

Let us show an example with \(n_u = 2\), \(n_z = 3\),
\(\gamma^z = [1, s, s]^\top\) and \(\gamma^v = [1, s]^\top\), where \( s \in \{-1,1\}\).
% \MK{ \(s_1\) vyhodit}
\begin{equation}
	\begin{bmatrix}
		b_1 & b_2 \\
		b_3 & b_4 \\
		b_5 & b_6 \\
	\end{bmatrix}
	\begin{bmatrix}
		1 & 0 \\ 
		0 & s
	\end{bmatrix}
	= 
	\begin{bmatrix}
		1 & 0 & 0 \\
		0 & s & 0 \\
		0 & 0 & s \\
	\end{bmatrix}
	\begin{bmatrix}
		b_1 & b_2 \\
		b_3 & b_4 \\
		b_5 & b_6 \\
	\end{bmatrix}
\end{equation}
\begin{equation}
	\begin{bmatrix}
		\bm{b_1} & sb_2 \\ 
		b_3 & \bm{sb_4} \\
		b_5 & \bm{sb_6} \\
	\end{bmatrix}
	= 
	\begin{bmatrix}
		\bm{b_1} & b_2 \\
		sb_3 & \bm{sb_4} \\
		sb_5 & \bm{sb_6} \\
	\end{bmatrix}.
\end{equation}
The bold elements have the same value for all \(s\),
we see that \(b_2,b_3\), and \(b_5\) have to be zero in order for the 
equation to hold.
 The matrix \(B\) will then be 
\begin{equation}
	B = \begin{bmatrix}
		b_1 & 0 \\
		0 & b_4 \\
		0 & b_6 \\
	\end{bmatrix}.
\end{equation}
\end{proof}

% \MK{Dal bych sem example, ktery ty symetrie ukaze od A do Z. Muzes pokracovat ten example  rovnice (26).}

% \MK{Corollary: co to znamena pro Koopmana}
\begin{corollary}
	\label{sec:corrolary}
	Assume a nonlinear system \(x^+ = f(x,u)\),
	 its predictor \(z^+ = Az + Bv, \hat{x} = Cz\)
	both with symmetry \(\Gamma^x\) according to \eqref{eq:symmetry_def} and \eqref{eq:symmetry_lti_def},
and symmetric lifting functions \(\Phi\) and \(\Psi\)
such that \(\gamma^z \Phi(x) = \Phi(\gamma^x x)\) and 
\(\gamma^v \Psi(u) = \Psi(\gamma^u u)\),
where 
\(\gamma^x \in \Gamma^x\), \(\gamma^v=\gamma^u = h(\gamma^x)\),
and \(\gamma^z = h'(\gamma^x)\).

	% If \((\hat{x}_t)_{t=0}^{\infty}\) is a prediction 
	% of \((x_t)_{t=0}^{\infty}\), then \((\gamma^x \hat{x}_t)_{t=0}^{\infty} = (\widehat{\gamma^x x})_{t=0}^{\infty}\) is a prediction of \((\gamma^x x_t)_{t=0}^{\infty}\) \newmk{achieving the same prediction error as \((\hat{x}_t)_{t=0}^{\infty}\) achieves for predicting \((x_t)_{t=0}^{\infty}\) }

	If \((\hat{x}_t)_{t=0}^{\infty}\) and \((\widehat{\gamma^x x})_{t=0}^{\infty}\)
	are predictions of the symmetrical trajectories \((x_t)_{t=0}^{\infty}\) and \((\gamma^x x_t)_{t=0}^{\infty}\) respectively,
	then \((\widehat{\gamma^x x})_{t=0}^{\infty}\) is equal to  
	symmetrized \((\hat{x}_t)_{t=0}^{\infty}\), i.e.
	\((\widehat{\gamma^x x})_{t=0}^{\infty} = (\gamma^x \hat{x}_t)_{t=0}^{\infty} \),
	achieving the same prediction error as \((\hat{x}_t)_{t=0}^{\infty}\) achieves for predicting \((x_t)_{t=0}^{\infty}\).

	% then \((\gamma^x \hat{x}_t)_{t=0}^{\infty}\) is also a prediction of 
	% \((\gamma^x x_t)_{t=0}^{\infty}\).

	% \MK{Can we prove that?}

	% The state of a nonlinear system \(x^+ = f(x,u)\) with 
	%  symmetry \(\Gamma^x\) according to \eqref{eq:symmetry_def},
	%  \corr{and the output of a linear Koopman predictor
	%  \(z^+ = Az + Bv\)  \(\hat{x} = Cz\)}

% 	 can be approximated 
% 	 by a linear Koopman predictor \(z^+ = Az + Bv\), such that 
% 	 its output \(\hat{x} = Cz\) respects the symmetry \(\gamma^x \in \Gamma^x\),
% 	 according to \eqref{eq:symmetry_lti_def}. \MK{The first statement is not a precise mathematcal statement. ``can be approxmated'' does not mean anything. Even a zero trajectory is an approximation.}

% 	 The lifting functions of the Koopman predictor \(\Phi\) and \(\Psi\) need to respect the symmetries so that 
% 	  \(\gamma^z \Phi(x) = \Phi(\gamma^x x)\) and 
% 	 \(\gamma^v \Psi(u) = \Psi(\gamma^u u)\),
% 	 where 
% 	 \(\gamma^x \in \Gamma^x\), \(\gamma^v=\gamma^u = h(\gamma^x)\),
% and \(\gamma^z = h'(\gamma^x)\).
\end{corollary}
\begin{proof}
	Let us have a nonlinear system \eqref{eq:nlsys} with symmetries 
	\eqref{eq:symmetry_def}. 
	The symmetrical states \(x_t, \gamma x_t\) of the system are obtained as 
	\begin{equation}
	\begin{alignedat}{2}
		x_t &= S_{f}^{t}(x_0, (u_0,\dots,u_{t-1}))\\
		\gamma^x x_t &= S_{f}^{t}(\gamma^x x_0, (\gamma^u u_0,\dots,\gamma^u u_{t-1})),
	\end{alignedat}
	\end{equation}
	where 
	\(S_{f}^{t}\) denotes the flow of the system \(f\) up to time \(t\) with initial condition \(x_0\) and input sequence \(u_{0},\dots,u_{t-1}\) .
	Using the 
	same initial state and inputs, the lifted states \(z_t, \gamma^z z_t\) of the LTI predictor \((A,B)\) are 
	\begin{equation}
		\label{eq:cor_zt}
	\begin{alignedat}{2}
		z_t &=  S_{A,B}^{t}(\Phi(x_0), (\Psi(u_0),\dots,\Psi(u_{t-1}))) \\ 
		\gamma^z z_t &=  S_{A,B}^{t}(\gamma^z \Phi(x_0), (\gamma^v \Psi(u_0),\dots,\gamma^v \Psi(u_{t-1}))).
	\end{alignedat}
	\end{equation}
	The prediction of \(x\) is obtained
	as \(\hat{x} = C z\), therefore 
\begin{equation}
\begin{alignedat}{2}
	\hat{x}_t &= C \cdotp S_{A,B}^{t}(\Phi(x_0), (\Psi(u_0),\dots,\Psi(u_{t-1}))) \\ 
	\widehat{\gamma^x x_t} 
		 &= C \cdotp S_{A,B}^{t}(\Phi(\gamma^x x_0), (\Psi( \gamma^u u_0),\dots,\Psi(\gamma^u u_{t-1}))).
\end{alignedat}
\end{equation}
Using the symmetries of the lifting functions
and definition of \(\gamma^z\) \eqref{eq:gammaz_def},
we can rewrite the second equation as 
\begin{equation}
\begin{alignedat}{2}
	\widehat{\gamma^x x_t} 
		 &= C \cdotp S_{A,B}^{t}(\gamma^z \Phi( x_0), (\gamma^u \Psi(  u_0),\dots,\gamma^u \Psi( u_{t-1}))) \\ 
&= C \gamma^z z_t \\
&= \gamma^x C z_t \\ 
&= \gamma^x \hat{x}_t.
\end{alignedat}
\end{equation}
\end{proof}

\paragraph*{Symmetric lifting functions}
\label{par:symmetry_psi}
The definition of \(\Phi\) from \eqref{eq:phidef}
can be extended to respect the symmetries by simply 
defining it on a symmetric dataset
\begin{align}
	\begin{split}
		\label{eq:phidef_sym}
		\bar{z}_0^i &= \Phi(\bar{x}_0^i)\qquad \forall i \in \mathbb{Z}_{1,2^{n_{\Gamma}}\cdot N},
	\end{split}
	\end{align}
	where 
	% \begin{equation}
	% \begin{alignedat}{2}
	% 	\bar{z}^i_0 &\in  \{\gamma^z z \quad \forall z \in Z_0, \forall \gamma^z = h'(\gamma^x) \} \\
	% 	\bar{x}^i_0 &\in  \{\gamma^x x \quad \forall x \in X_0, \forall \gamma^x \in \Gamma^x \}.
	% \end{alignedat}
	% \end{equation}
	\begin{equation}
		(\bar{x}^i_0,\bar{z}^i_0) \in 
		\{
			(\gamma^x x,\gamma^z z):
			\gamma^z = h(\gamma^x) \quad \forall \gamma^x \in \Gamma^x
		\}
	\end{equation}
	Since the input transformation \(\Psi\) is done channel-wise,
	for each \(\Psi_k\) we can simply add the cost
	\begin{align}
		\begin{split}
			\sum_{j\in \mathbb{Z}_{1,q_k}} \sum_{u^j_k \in U_k v^j_k \in V_k} || \gamma^v v^j_k - \Psi_k(\gamma^u u^j_k)||_2^2 
		\end{split}
	\end{align}
	to the regularizations \(\theta(\cdot)\) in \eqref{eq:problem_description}.

\section{Implementation details}
\label{sec:implementation}
\newcor{This section describes certain details concerning the 
problem \eqref{eq:problem_description}. 
First, we address the trajectory preparation in Section \ref{sec:trajectory_preparation}. 
Solving of the problem \eqref{eq:problem_description} and its initialization are discussed in Sections \ref{sec:solving_the_problem} and 
\ref{sec:initvals} respectively.

A summary of the whole process from data preparation 
to solving \eqref{eq:problem_description} is provided in the Section \ref{sec:summary_learning}.
}

\subsection{Trajectory preparation }
\label{sec:trajectory_preparation}
\newcor{In order to improve the prediction capabilities and restrict
overfitting, we force consecutive trajectories to 
have the same lifted values in their endpoints. This
\textit{endpoint consistency} it is enforced by the term 
$\sum_{(a,b)\in \mathcal{S}} ||z^a_{H_{\rm T}} - z^{b}_{0}||_2^2$
of \eqref{eq:problem_description}. The effect of (not) using 
it is shown in the Fig.\ref{fig:traj_endpoints}.
In our examples, we enforce the endpoint consistency and the prediction capabilities of the predictor exceed the learning horizon $H_{\rm T}$. This is shown in the Fig.\ref{fig:car_openloop} \corplane{in the numerical examples section}.

To make sure that our dataset contains consecutive 
trajectories, we generate long trajectories of 
length \(r  H_{\rm T}\) and split each into
\(r\) trajectories of length \(H_{\rm T}\). 
Hence the final state of the first trajectory is the initial state of the second one and so on (only the last trajectories will have ``free" final states \(x_{rH_{\rm T}}\)). 
The indices of consecutive trajectories are stored in the set $\mathcal{S}$, as depicted in the Fig.\ref{fig:traj}.
}

% In order to ensure that the lifted trajectories will not end outside of the set \(Z_0\),
% we will assume that the final states of some trajectories 
% are also in the set \(Z_0\). An example can be seen in the Figure \ref{fig:traj_endpoints}.
% This will be done simply by generating long trajectories of 
% length \(r  H_{\rm T}\) and splitting each into
% \(r\) trajectories of length \(H_{\rm T}\). 
% Hence the final state of the first trajectory is the initial state of the second one and so on (only the last trajectories will have free final states, not included in \(Z_0\) ).
% The pairs of connected trajectories can be formalized as
% \begin{align}
% \begin{split}
% 	\mathcal{S} := \{(i,j) \in (\mathbb{Z}_{1,N} \times \mathbb{Z}_{1,N}) :
% 	y_{N}^i = y_0^j, i\neq j \},
% \end{split}
% \end{align}
% where \(\mathcal{S}\) contains the indices of connected trajectories. 
 % The possible cases of a lifted trajectory with relation to \(Z_0\) are depicted in the Fig.\ref{fig:traj_endpoints}.
 % Note that we could constrain all the points of a trajectory to be inside \(Z_0\), to prevent the case \(\mathcal{T}_3\) from Fig.\ref{fig:traj_endpoints},  but it is not considered in this work.
\begin{figure}[!htp]
	\centering
	\includegraphics{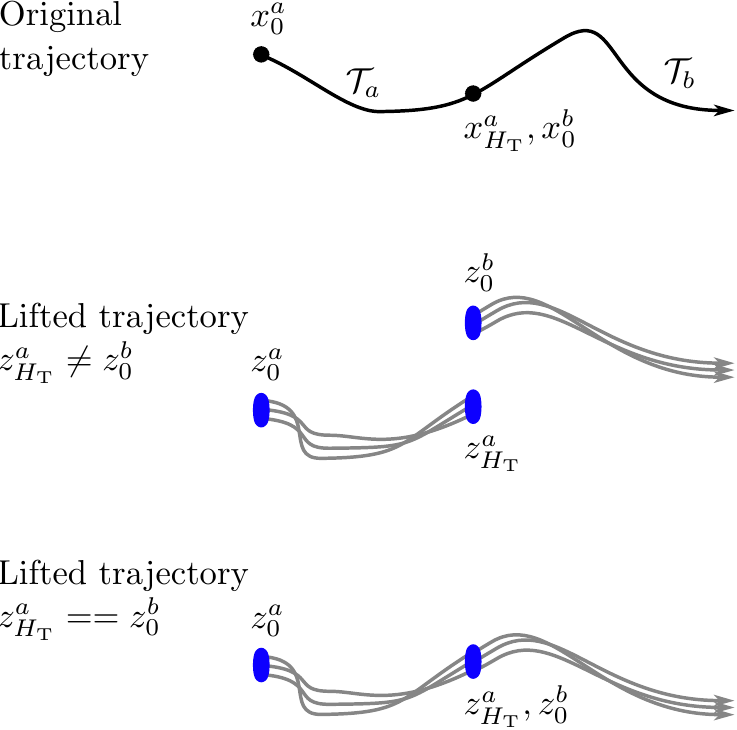}
	\caption{
\newcor{This figure demonstrates the effect of enforcing the endpoint consistency. Note that even when not enforcing it,
we still obtain a valid predictor although the prediction capabilities
will be strictly limited to the \mbox{learning horizon $H_{\rm T}$}.
}
}
		\label{fig:traj_endpoints}
\end{figure}
\begin{figure}[!ht]
	\centering
	\includegraphics{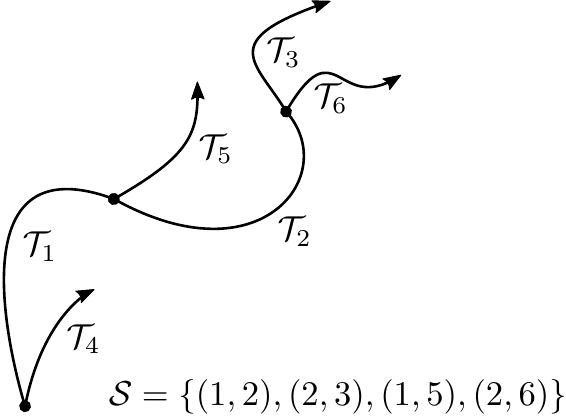}
	\caption{Example of trajectories used for learning and their interconnections. We see that one long trajectory was split into 
 $\mathcal{T}_1$, $\mathcal{T}_2$, and $\mathcal{T}_3$. The trajectories 
 $\mathcal{T}_{4,5,6}$ start from the same initial conditions as 
 $\mathcal{T}_{1,2,3}$ in this order.
 The set $\mathcal{S}$ contains the indices of the interconnected trajectories, denoting that 
 the pairs of connected trajectories are \((\mathcal{T}_1,\mathcal{T}_2)\),\((\mathcal{T}_2,\mathcal{T}_3)\), etc.}
		\label{fig:traj}
\end{figure}

Since the problem \eqref{eq:problem_description} is very flexible in terms of free variables, we need to make 
a clear distinction between controlled
and autonomous trajectories, to prevent overfitting (we do not want controlled trajectory to be approximated by autonomous response and vice versa).
For this, it is sufficient to start two trajectories 
from the same initial state, with different control inputs. This is also depicted in Fig.\ref{fig:traj}.
We state this formally in the following Lemma.
% -----------------
{\color{black}
% \newpage
% \begin{lemma}
% 	The trajectories of \(f(x,u)\) with identical initial states and unique control inputs 
% 	will only be approximated by trajectories of the Koopman predictor with unique lifted inputs, assuming that the Koopman predictor fits the trajectories perfectly.
% \end{lemma}
\begin{lemma}
	\label{sec:dataset_control_separation}
	Assume that we have two distinct trajectories of 
	the nonlinear system \eqref{eq:nlsys}, \(\mathcal{T}_i\) and \(\mathcal{T}_j\), with \(x_0^i = x_0^j\), \(u_0^i \neq u_0^j\), and \(y_1^i \neq y_1^j\). 
	For any Koopman predictor \(\mathcal{K}\), which approximates the trajectories with zero error, 
	\COR{such that \(z_0^i = z_0^j\),
	\(z_0^{i} = \Phi(x_{0}^{i})\),
	\(z_0^{j} = \Phi(x_{0}^{j})\) , 
	\(Cz_1^{i} = y_1^{i}\),
	\(Cz_1^{j} = y_1^{j}\), }
	it holds that \(\Psi(u_0^i) \neq \Psi(u_0^j)\).

	% Assume that we have two distinct trajectories of 
	% the nonlinear system \eqref{eq:nlsys}, \(\mathcal{T}_i\) and \(\mathcal{T}_j\), with \(x_0^i = x_0^j\), \(u_0^i \neq u_0^j\), and \(y_1^i \neq y_1^j\). 
	% For any Koopman predictor \(\mathcal{K}\), which approximates the trajectories with zero error, such that \(z_0^{i/j} = \Phi(x_{0}^{i/j})\), \(z_0^i = z_0^j\), and \(Cz_1^{i/j} = y_1^{i/j}\), it holds that \(\Psi(u_0^i) \neq \Psi(u_0^j)\). 
	% \MK{Co znamenaji ty lomitka v supscriptu? Tohle nejsou tvoje poznamky!}
\end{lemma}
\begin{proof}
	Let us have two trajectories 
	\begin{equation}
	\begin{alignedat}{2}
	\mathcal{T}_i: &(x_0^i, y_0^i, u_0^i),(y_1^i) \\ 
	\mathcal{T}_j: &(x_0^j, y_0^j, u_0^j),(y_1^j) 
	\end{alignedat}
	\end{equation}
	such that \(i \neq j\), \(u_0^i \neq u_0^j\), and \(y_1^i \neq y_1^j\).	
	We want to show that if \(x_0^i = x_0^j\) then \(\Psi(u_0^i) \neq \Psi(u_0^j)\).
 
	Let us write the predictions of the Koopman predictor
	\begin{equation}
	\begin{alignedat}{2}
		\hat{y}_1^i &= C(A z_0^i + B v_0^i) \\
		\hat{y}_1^j &= C(A z_0^j + B v_0^j),
	\end{alignedat}
	\end{equation}
	% where \(z_0^{i/j} = \Phi(x_{0}^{i/j})\) and \(v_0^{i/j} = \Psi(u_{0}^{i/j})\).
	\COR{  where 
	\(z_0^{i} = \Phi(x_{0}^{i})\),
	\(z_0^{j} = \Phi(x_{0}^{j})\),
	\(v_0^{i} = \Psi(u_{0}^{i})\),
	\(v_0^{j} = \Psi(u_{0}^{j})\)  }
	By subtracting them, we obtain
	\begin{equation}
		\hat{y}_1^i - \hat{y}_1^j = CA(z_0^i - z_0^j) + CB(v_0^i - v_0^j).
	\end{equation}
	Knowing that we have zero approximation error, 
	we can write 
	% \(\hat{y}_{1}^{i/j} = y_1^{i/j}\)
	\COR{\(\hat{y}_{1}^{i} = y_1^{i}\) and \(\hat{y}_{1}^{j} = y_1^{j}\)}
	.
	Then by using the inequality
	\({y}_1^i \neq {y}_1^j \), we get 
	\begin{equation}
		CA(z_0^i - z_0^j) + CB(v_0^i - v_0^j) \neq 0.
	\end{equation}
	The inequality can be fulfilled by either 
	the initial states or the inputs not being equal.
	However, if both trajectories start from the same initial condition, we get \(z_0^i = z_0^j\) and the inequality can be only satisfied by \(v_0^i \neq v_0^j\).
% \MK{Hele, vzdyt tady rank of $CB$ nehraje roli.}
\end{proof}
}

% -----------------

% \begin{alignedat}{2}
% 	x^a_0 &= \begin{bmatrix}
% 		2\\3
% 	\end{bmatrix}, &x^a_1 = &\begin{bmatrix}
% 		5\\1.5
% 	\end{bmatrix}\\
% 	x^b_0 &= \begin{bmatrix}
% 		2\\-3
% 	\end{bmatrix}, &x^b_1 = &\begin{bmatrix}
% 		5\\-1.5
% 	\end{bmatrix}.
% \end{alignedat}
% \end{equation}
% The symmetry is caused by \(x_2\) not being depended on \(x_1\) (by the zero in the A matrix).

% We can use this principle for making the Koopman predictor symmetric. 
% Assume a lifted space \(z = \begin{bmatrix}
% 	z^1 & z^2
% \end{bmatrix}^\top\), by fixing the sparsity pattern of \(C\) as follows:
% \begin{equation}
% 	C = \begin{bmatrix}
% 		c_{1,1} & \dots & c_{1,n_{z^1}} & 0 & \dots & 0\\
% 		0 & \dots & 0 & c_{2,n_{z^1}+1} & \dots c_{2,n_{z^1} + n_{z^2}},
% 	\end{bmatrix}
% \end{equation}
% we ensure that the lifted vectors 
% \(z^{1},z^2\) correspond to the original states \(x^{1},x^2\).
% Assuming symmetry around \(x_1\), the Koopman matrix \(A\) is 
% \begin{equation}
% 	A = \begin{bmatrix}
% 		A_{1,1} & A_{1,2}\\
% 		\boldsymbol{0} & A_{2,2}
% 	\end{bmatrix}.
% \end{equation}

% We can also decrease the number of parameters of the optimization process. Instead of having both \(z^a\) and \(z^b\) as parameters,
% we can write 
% \begin{equation}
% 	z^b := \begin{bmatrix}
% 		\boldsymbol{1}_{n_{z^a}} \\ \boldsymbol{-1}_{n_{z^b}} 
% 	\end{bmatrix} z^a
% \end{equation}
% and keep only \(z^a\) as a parameter.
\subsection{Solving the Koopman optimization problem}
\label{sec:solving_the_problem}
\corplane{The problem \eqref{eq:problem_description} can be formulated as 
a nonlinear unconstrained optimization problem after elimination 
of the linear equality constraints. It can be solved by a variety of solvers; in our case, we use the 
\mbox{ADAM \cite{Kingma2014}} optimization algorithm.
ADAM is a first-order method for unconstrained problems,
therefore it requires gradients of the cost function.
We calculate the gradients of \eqref{eq:problem_description},
via automatic differentiation (AD) routines, specifically by
the packages Flux \cite{innes:2018},\cite{Flux.jl-2018} and Zygote
 \cite{Zygote.jl-2018} from  Julia \cite{bezanson2017julia}.
}

% The problem \eqref{eq:problem_description} is solved using unconstrained \mbox{ADAM \cite{Kingma2014}} optimization algorithm.
% ADAM is a first-order method for unconstrained problems, whereas our problem has equality constraints and 
% we do not provide any gradients.
% The equality constraints can be simply substituted into
% the cost function. 
% In order to obtain the gradients of \eqref{eq:problem_description},
% we use automatic differentiation (AD) routines, specifically
% the packages Flux \cite{innes:2018},\cite{Flux.jl-2018} and Zygote
%  \cite{Zygote.jl-2018} from the language Julia \cite{bezanson2017julia}. 
%  \MK{References not compiled properly}
 
We would like to note that it seems beneficial to 
fix the values \(V_k\) during the first iterations of the 
ADAM solver. The lifted inputs will 
be more likely to retain the physical meaning of the 
original input variables. 
This is demonstrated in our last example in the Figures \ref{fig:veh_lift_1} and \ref{fig:veh_lift_2}. In this paper, we fixed the input for the first 500 iterations.

\subsection{Initialization values}
\label{sec:initvals}
The variables of \eqref{eq:problem_description}
need to be initialized. Table \ref{table:init}
lists the choices used in this paper. 
\COR{We use \(\mathcal{U}_{\alpha,\beta}\) to denote uniform distribution in the }interval \([\alpha,\beta]\).
% \MK{What is $\mathcal{U}$?}
\begin{table}[h!]
\centering
	\begin{tabular}{|c|c|} 
	\hline
	Variable & Initialization value(s) \\ 
	\hline
	\(z_0^i\) & Each element from \(\mathcal{U}_{\nicefrac{-1}{2},\nicefrac{1}{2}}\) \\ 
	\(V_k\) & Same as \(U_k\) \\
	% \(A\) & {\color{red} \(a_{i,j} \in \mathcal{U}_{0, \nicefrac{1}{n_z}}\) }\\
	\(A\) & \(a_{i,j} \in \mathcal{U}_{\nicefrac{-1}{2},\nicefrac{1}{2}}\) \\
	\(B\) & \(b_{i,j} \in \mathcal{U}_{\nicefrac{-1}{2},\nicefrac{1}{2}}\) \\
	\(C\) & \(c_{i,j} \in \mathcal{U}_{\nicefrac{-1}{2},\nicefrac{1}{2}}\) \\
	\hline
\end{tabular} 
\caption{Initialization values for the problem \eqref{eq:problem_description}.}
\label{table:init} 
\end{table}

It might be beneficial in some cases to initialize \(A\) such 
that its spectral radius is less than 1, which 
will make the initialized system \((A,B,C)\) stable from a control-systems point of view.

\COR{\subsection{Summary: Learning Koopman predictor}
\label{sec:summary_learning}
The algorithm \ref{alg:koopalg} summarizes the procedure for learning the Koopman predictor.

\begin{algorithm}
	\caption{Find the Koopman predictor (A,B,C,\(\Phi,\Psi\))}
	\label{alg:koopalg}
	\begin{algorithmic}[1]
	\Require $f,g ,X_0,H_{\rm T},q_k,N$
	% \rednote{TODO nepotrebuju f,g pro symetrie?
	% Taky data \(\mathcal{D}\) musim generovat specifickym zpusobem (Fig.\ref{fig:traj}), to by melo byt soucasti toho algritmu a ne jako input. Leda to rozdelit na dva? Data generation and Finding the predictor? } 
 % \MK{Pro symetrie ano, ale ty data bys takto mohl generovat i bez znalosti dynamiky. Pro initial submission to takle nech a jen smaz ten koment.}
		\State Identify symmetries of the system \(f,g\) according to \ref{sec:symmetry} and create the sparsity patterns for \(A,B,C\).
		\State Choose the number of quantization levels \(q_k\) and create the quantized channels \(U_k\).
		\State Generate \(N\) trajectories \(\mathcal{T}_i\) of length \(H_{\rm T}\) according to \ref{sec:trajectory_preparation}
		\State Initialize the matrices \(A,B,C\), and the elements of \(Z_0\) and \(V_k\) according to \ref{sec:initvals}.
		\State Solve \eqref{eq:problem_description}, optionally fix the values \(V_k\) as mentioned in \ref{sec:solving_the_problem}.
		\Ensure $A,B,C,Z_0,V_k$
	% \EndProcedure
	\end{algorithmic}
\end{algorithm}}
% \MK{Jednovetovy sekce prosim ne... + }

\section{Koopman operator in control}
\label{sec:control}
In this section, we shall demonstrate the advantage of using a Koopman operator in control applications. More specifically, we will consider a Model Predictive Control (MPC) \COR{\cite{Mayne2000}}.
%  \MK{Cite a standard reference like Mayne 2000.} 
 For a general nonlinear system \eqref{eq:nlsys}, the MPC can be formulated as 
\begin{equation}
	\label{eq:mpc_nl}
	\begin{aligned}
	J^{\star} = \min \quad & \sum_{t=1}^{H} J_{\rm n}(x_t) + J_{\rm c}(x_t) \\
	\textrm{s.t.} \quad & x_{t+1} = f(x_t,u_t)\\
	& x_t \in X \\
	&  u_t \in U,
	\end{aligned}
\end{equation}
where \( J_{\rm c}\) and \(J_{\rm n}\) are convex and non-convex parts of the cost function in this order. 
This problem is non-convex due to the function \(J_{\rm n}\) and the generally nonlinear constraint \( x_{t+1} = f(x_t,u_t)\); this makes
the problem difficult to solve in general. We can, however, use  
the Koopman methodology to reformulate the problem in a convex fashion.
We can eliminate the non-convex cost \(J_{\rm n}(x_t)\) by setting it as an additional output \( \chi_t = J_{\rm n}(x_t)\) \newmk{and use} the Koopman predictor of the form 
% \MK{The symbol $\gamma$ denoted the symmetry action before!!!}
\begin{equation}
\begin{alignedat}{2}
 	z_{t+1} &= Az_t + B v_t \\
	\begin{bmatrix}
		\hat{\chi}_{t} \\ \hat{x}_t
	\end{bmatrix} &= C z_t\\
	z_0 &= \Phi(x_0),
\end{alignedat}
\end{equation}
where \(\hat{\chi}_t\) is an approximation of the non-convex cost \(\chi_t = J_{\rm n}(x_t)\).
We can now rewrite \eqref{eq:mpc_nl} as
\begin{equation}
	\label{eq:mpc2}
	\begin{aligned}
	\hat{J}^{\star} = \min \quad & \sum_{t=1}^{H} \hat{\chi}_t + J_{\rm c}(\hat{x}_t) \\
	\textrm{s.t.} \quad & z_{t+1} = Az_t + B v_t\\
	& \begin{bmatrix}
		\hat{\chi}_{t} \\ \hat{x}_t
	\end{bmatrix} = C z_t\\
	& z_0 = \Phi(x_0) \\ 
	& \hat{x}_t \in X \\
	& v_t \in V,
	\end{aligned}
\end{equation}
where \(X\) is assumed to be convex and \(V\) is a box-constraint by construction 
(due to the channel-wise lifting). Under these assumptions, the problem \eqref{eq:mpc2} is convex.

In a wide range of applications, \COR{the convex part of the cost is quadratic and the constraints are upper and lower bounds on some of the variables}. This allows the user to 
formulate the problem~(\eqref{eq:mpc2}) as a convex Quadratic Program (QP) -- a
well-studied class of convex optimization problems with many solvers tailored to solving them efficiently such as 
 ProxSuite \cite{bambade:hal-03683733}, OSQP \cite{osqp}, and COSMO \cite{Garstka_2021}.
 \subsection{Koopman MPC}
 Let us exploit the benefits of the QP formulation and formulate our problem concretely, in a tracking form (minimizing the deviation of \(y_t\) from known \(y_{\rm ref}\) ) consistent with our own implementation
used later in the Section \ref{sec:numerical_examples}. 
We shall refer to this optimization problem as Koopman MPC (KMPC):

\begin{equation}
	\label{eq:MPC2QP}
	\begin{alignedat}{2}
	\min_{\Delta v_t, z_t} \quad & \sum_{t=1}^{H}
	||y_{\rm ref} - C z_t||^2_{Q} + ||v_t||^2_{R} &&+||\Delta v_t||^2_{R_{\rm d}} \\ 
	\textrm{s.t.} \quad & z_{t+1} = A z_t + B v_t && t = 1\ldots H-1\\
	& v_t = \Delta v_{t} + v_{t-1} && t = 0\ldots H\\
	& y_{\rm low} \leq y_t \leq y_{\rm up} && t = 1\ldots H\\
	& v_{\rm low} \leq v_t \leq v_{\rm up} && t = 1\ldots H\\
	& \Delta v_{\rm low} \leq \Delta v_t \leq \Delta v_{\rm up} && t = 1\ldots H\\
	& z_0 = \Phi(x_0) \\ 
	& v_0 = \Psi( u_{\rm prev}),
	\end{alignedat}
\end{equation}
where \(y_t = \begin{bmatrix}
	\hat{\chi}_t \\ \hat{x}_t
\end{bmatrix}\), \(Q \succcurlyeq 0\), \(R \succ 0\), and \(R_{\rm d}\succ 0\).
As mentioned before, this formulation solves the tracking problem where \(y_{\rm ref}\) is an external parameter, and 
 \(\Delta v_t\) is the optimization variable (instead of \(v_t\)). 
\COR{The output, input, and input rate constraints are the pairs \((y_{\rm low},y_{\rm up})\), \((v_{\rm low},v_{\rm up})\), and \((\Delta v_{\rm low},\Delta v_{\rm up})\) in this order. }
% \MK{Zas ty lomitka.}

The optimal 
solution is recovered as \(u_{t}^{\star} = \hat{\Psi}^{-1}(v^{\star}_t) \),
\COR{where \(\hat{\Psi}\) is linearly interpolated version of \(\Psi\)
which is discussed in detail the following Section.}
The problem is solved repeatedly at each time step, always using only 
the first control input \(u_{1}^{\star}\) and then recalculating the 
solution from a new initial state.
 This approach provides closed-loop control and 
 can be seen in Fig. \ref{fig:mpc_scheme}.

 \begin{figure}[!ht]
	\centering
		\includegraphics{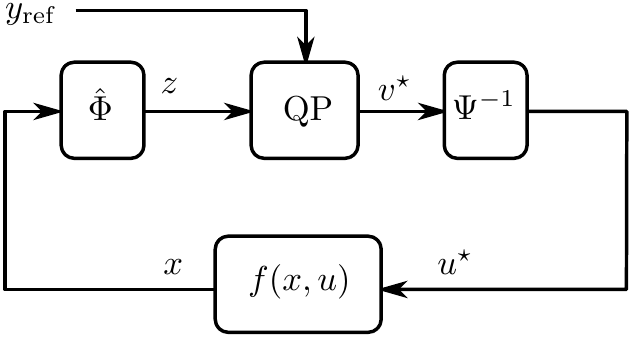}
	\caption{General MPC scheme using the Koopman operator as 
	a control design model.}
		\label{fig:mpc_scheme}
\end{figure}
 
%  \MK{Algorithm summary nekde.}
\corr
 {\subsection{Control-related considerations}}
 For use in control, we need to expand the domain of the lifting
 function \(\Phi\) \eqref{eq:phidef} to \(X\), instead of \(X_0\).
 Another matter to consider is the invertibility of \(\Psi\) \COR{\eqref{eq:psidef}}, 
 % since we need to transform \(\mathbf{v}^{\star}\) into \(\mathbf{u}^{\star}\)
 % by \(\mathbf{u}^{\star} = \Psi^{-1}(\mathbf{v}^{\star})\) in order to control the nonlinear system \eqref{eq:nlsys}.
 since it is not guaranteed from \eqref{eq:problem_description}.
% \COR{ We may also drop the quantization of the input by interpolating 
%  \(\Psi\).}

 Both of these matters are addressed in the following sections
 \ref{sec:interpolation_phi} and \ref{sec:invertibility_psi}.

 Regarding the connection of the Koopman predictor and MPC, 
we can expect the lifting functions to have a non-zero lifting error since
we are working only with an approximation of the Koopman operator.
The lifting error can be calculated and its knowledge exploited as a part of the MPC algorithm. This option is addressed in \ref{sec:lifting_error}.

The last consideration is setting of the lifted input rate bounds \(\Delta v_{\rm up/low}\). These bounds do not naturally arise by finding the lifting functions (unlike \(v_{\rm up/low}\)) and need to be set manually. 
We address this further in \ref{sec:inputrate}.

 Lastly, we note that the problem \eqref{eq:MPC2QP} has 
 the lifted state vectors \(z_t\) as variables. 
 This greatly increases the total number of optimization variables
 since the dimension \(n_z\) \corplane{may be large} (recall that 
 we are approximating infinite-dimensional operator).
 In the appendix, we show that the MPC can be formulated in a so-called condensed formulation, which has only \(\Delta v_t\) as variables,
 reducing the computational burden. 
%  \begin{enumerate}
% 	\item Exploiting the knowledge of the lifting error
% 	\item Setting of the lifted input rate bounds \(\Delta v\)
% 	\item Interpolation of \(\Phi\)
% 	\item Invertibility of \(\Psi\)
% 	\item Condensed form of the MPC
% \end{enumerate}

 \subsubsection{Lifting error}
 \label{sec:lifting_error}
% Note that we consider output disturbance \(d\) in the formulation \eqref{eq:MPC2QP}.
When the lifting via \(\hat{\Phi}\) is not exact, we can use the knowledge of \(\hat{\Phi}\) to calculate the lifting error at \(x_{\rm init}\)
and set it as output disturbance to the MPC to increase 
the precision of the calculation for the first timestep.
The augmented output equation \corplane{then reads}
\begin{equation}
	y_t = C z_t + d_t,
\end{equation}
where \(d\) is the initial-state lifting error
\begin{equation}
	\label{eq:disturbance_mpc}
	d_t = g(x_{\rm init}) - C\hat{\Phi}(x_{\rm init}) \COR{\quad \forall t = 1,\dots,H.}
\end{equation}
% \MK{Either $d$ is assumed constant or used the subscript $d_k$.}
This can be easily implemented by augmenting the state-space model as 
\begin{equation}
\begin{alignedat}{2}
\bar{A} = \begin{bmatrix}
	A & \\
	& \zeta I
\end{bmatrix}, \bar{B} = \begin{bmatrix}
	B \\ 0
\end{bmatrix},
C = \begin{bmatrix}
	C & I
\end{bmatrix},
\end{alignedat}
\end{equation}
where \(|\zeta| \leq 1 \) is the decay rate of \(d\).
Constant disturbance model is achieved by setting \(\zeta = 1\).

The error could be also estimated for the whole prediction horizon iteratively, by solving the QP multiple times and evaluating \eqref{eq:disturbance_mpc} for all the states \(x_t\).

% The last improvement to consider is applicable in the special case, where we 
%  measure and track the state vector directly.
%  We may set the tracking reference as \(C\Phi(x_{rm ref})\) instead of \(x_{\rm ref}\).

 \subsubsection{Input rate bounds}
 \label{sec:inputrate}
The bounds on \(\Delta v\) are not trivial to choose because
the lifted-space bounds on \(v\) will not correspond to the real bounds on \(u\) 
\COR{because of the nonlinearity of $\Psi$}, as seen in Fig.\ref{fig:UV_graf}.
Possible solutions are
\begin{enumerate}
	\item use the bounds with soft constraints to make them flexible
	\item use ideas from nonlinear MPC and make an iterative scheme (i.e. solve the QP \corplane{multiple times} while iteratively adjusting the bounds.)
	\item use linearization of \(\Psi\) to set the bound 
	\item  set the bounds on \(\Delta v\) conservatively, so that the worst case of \(\Delta u\) is guaranteed to be within its bounds.
\end{enumerate}

\begin{figure}[!ht]
	\centering
	\includegraphics{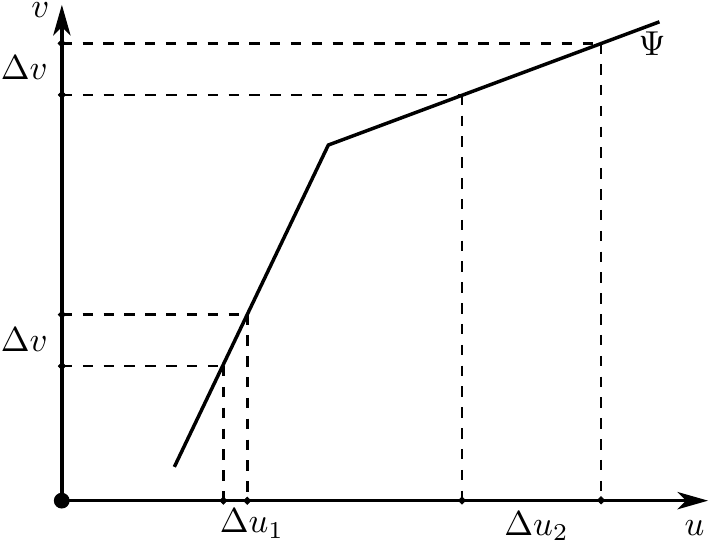}
	\caption{Example of a function \(\Psi\), which would pose challenges
	 with identifying the desired bounds on \(\Delta v\).
	 We see that constant \(\Delta v\)
	 can correspond to both \(\Delta u_1\) and \(\Delta u_2\) depending 
	 on the current value of the input.
	%  We can see that if we set constant bound on \(\Delta v\) on the y-axis,
	%  they will be reflected by varying bounds in \(\Delta u\)
	%  when inverting the function.{\color{red}	TODO porovnat}
	% We can see that the rate of \(u\) changes based on the slope
	% of \(\Psi\), even though the bound on \(\Delta v\) is constant.
	}
		\label{fig:UV_graf}
\end{figure}

% \begin{figure}[!htp]
% 	\centering
% 	\includegraphics[width=0.5\textwidth]{mpc_scheme.jpg}
% 		% \includegraphics{tikz/mpc-scheme.eps}
% 	\caption{General MPC scheme using the Koopman operator as 
% 	a control design model.}

% 		\label{fig:mpc_scheme}
% \end{figure}

\subsubsection{Interpolation of \(\Phi\)} 
% \MKres{Tohle by mohlo byt v sekci 5. VC: myslim ze ted se to sem hodi vic, tahle sekce je (po prejmenovani) o control a sekce 5 je ciste uceni se Koopmana.}
\label{sec:interpolation_phi}
\COR{As a result of the optimization process, we will obtain 
% We will obtain \MK{From where?} 
the samples of the function \(\Phi\) in the form of pairs \((x_0^i,z_0^i)\),}
instead of the function itself. 
One needs to approximate \(\hat{\Phi}\) by interpolation as
\begin{equation}
	\hat{\Phi}(x,p) = h(X_0,Z_0,x,p),
\end{equation}
where \(h\) is \corplane{an} interpolation method with parameters \(p\), e.g.,  the K-Nearest Neighbours (k-NN).
The question is how to choose the parameters \(p\)?
One way would be to simply evaluate the lifting error  across all datapoints and select the best one as 
\begin{equation}
	\label{eq:static_knn}
	 p^{\star} = \argmin_{p \in P} \sum_{x \in X_0} ||g(x) - C\hat{\Phi}(x,p)||^2_Q,
\end{equation}
where \(P\) is the parameter space.
\COR{We use the \(|| \cdot ||_Q\) norm
 to make the weighting consistent with the KMPC cost function in \eqref{eq:MPC2QP}.}
%   \MK{What is $Q$?}
Another possibility is to adapt the interpolation scheme dynamically to the current initial point \(x_{\rm init}\)
of the MPC as
\begin{equation}
	\label{eq:adaptive_knn}
	p^{\star} = \argmin_{p \in P} ||g(x_{\rm init}) - C\hat{\Phi}(x_{\rm init},p)||^2_Q.
\end{equation}
Doing this would ensure the precision of the first few steps of the prediction, as well as the precision of the first control input, which is used for the closed-loop control.
\COR{The tradeoff is that every iterate of the closed-loop would involve
 searching for the best interpolation parameters either via 
solving an optimization problem, or simply by evaluating the lifting function \(|P|\) times (if the search space \(P\) is finite, \corplane{such as for k-NN}).}

% The tradeoff is evaluating the lifting function \(|P|\) times at every timestep, therefore the viability of this scheme depends
% on the interpolation method and the real-time constraints. \MK{Not necessarily evaluating the interpolation $|P|$ times but resolving the optimization problem searching for best interpolation parameters at every iterate of closed-loop; note that the search space $P$ may be infinite.}

\subsubsection{Invertibility of \(\Psi\)}
% \rednote{VC: kdyz je to kvantovany tak porad muzou mit dva samply stejnou hodnotu. Ale realne to stejne chci interpolovat a mit spojitou funkci ne?} 
% \MK{Ano, invertovatelny to byt nemusi, ale nesouhlasim, ze chces nutne spojitou funkci. Nekdy to rizeni je priozene diskretni, treba razeni v aute.}
\label{sec:invertibility_psi}
The function $\Psi$ might be required to be invertible if we intent to 
use its continuous interpolation in the MPC.
% The function \(\Psi\) is not guaranteed to be invertible 
% from the problem \eqref{eq:problem_description}.

In order to ensure this, we can either manually limit the domain
of each channel to its invertible parts, which boils down to manually limiting the domain of \(n_u\) scalar functions of one variable.

Another option is enforcing monotonicity of individual 
\(\Psi_k\) 
\corplane{by adding the following regularization as the additional
cost term \(\theta(\cdotp)\) in \eqref{eq:problem_description}: 
\begin{equation}
	\label{eq:monotonicity}
	\left\lVert  \left|  v_k^{q_k} - v_k^1  \right| - \sum_{i=1}^{q_k - 1} | v_k^{i+1} - v_k^i |  \right\rVert_2^2,
\end{equation}
where \(v_k\) are samples of \(\Psi_k\). The cost \eqref{eq:monotonicity} is explained in the following Lemma.}
\begin{lemma}
\corplane{A sequence of \(q_k\) consecutive samples \([v_k^1,\dots,v_k^{q_k}]\) is monotonous if and only if }
% \MK{The statement of the Lemma is vague.}
\begin{equation}
	\label{eq:monotonicitylemma}
	\left|  v_k^{q_k} - v_k^1  \right| = \sum_{i=1}^{q_k - 1} | v_k^{i+1} - v_k^i |.
\end{equation}
\end{lemma}
\begin{proof}
 The final element of the sequence
 can be written as 
 \begin{equation}
	v_k^{q_k} = v_k^1 + \sum_{i=1}^{q_k-1} \epsilon_i,
 \end{equation}
 where \(\epsilon_i = v_k^{i+1} - v_k^i\).
We reorganize the terms
\begin{equation}
	v_k^{q_k} - v_k^1 = \sum_{i=1}^{q_k-1} \epsilon_i
 \end{equation}
 and put both sides of the equation in absolute value
 \begin{equation}
	\label{eq:proof_absfun}
	|v_k^{q_k} - v_k^1| = \left|\sum_{i=1}^{q_k-1} \epsilon_i \right|.
 \end{equation}
If the function is monotonous, all the \(\epsilon_i\) 
have the same sign and hence
\begin{equation}
	\label{eq:proof_monoabs}
	\left|\sum_{i=1}^{q_k-1} \epsilon_i \right| = \sum_{i=1}^{q_k-1} |\epsilon_i|.
\end{equation}
Finally, by assuming the monotonicity of \(\Psi_k\), we can put \eqref{eq:proof_absfun} and \eqref{eq:proof_monoabs} together and we obtain
\begin{equation}
	\label{eq:final_equality}
	|v_k^{q_k} - v_k^1| = \sum_{i=1}^{q_k-1} |\epsilon_i| = \sum_{i=1}^{q_k - 1} | v_k^{i+1} - v_k^i | .
 \end{equation}
 Therefore, if the sequence \(v^i_k\) is monotonous, the equality \eqref{eq:monotonicitylemma} must hold.\\
%  \MK{You are proving an if and only if statement, but the last sentence proves only one direction.}

\COR{Let us now prove the other direction. We claim that 
if \eqref{eq:monotonicitylemma} holds, then the function
is monotonous.
In the simple cases where the terms inside the 
absolute values are either all non-negative or non-positive,
it is trivial to see that the function will be non-decreasing or non-increasing respectively.

The interesting case is where the signs are different.
Let us assume, without loss of generality, that
\(v^2 - v^1 \leq 0\), and all the other terms are nonnegative.
If we rewrite \eqref{eq:final_equality} without the absolute values,
we obtain 
\begin{equation}
	(v_k^1 - v_k^2) + (v_k^3 - v_k^2) + (v_k^4 - v_k^3) + \dots + (v_k^{q_k} - v_k^{q_k - 1}) =
	v_k^{q_k} - v^1.
\end{equation}
Is is clear that
\begin{equation}
	v_k^1 - v_k^2 = 0,
\end{equation}
therefore \(\Psi_k\) will be monotonically non-decreasing.
The same approach can be used for multiple sign changes, which concludes the proof.}
\end{proof}

\COR{\section{Summary: Koopman MPC}
\label{sec:summary_mpc}
The algorithm \ref{alg:mpcalg} summarizes the procedure for designing the KMPC.
\begin{algorithm}[!htbp]
	\caption{Create the Koopman MPC}
	\label{alg:mpcalg}
	\begin{algorithmic}[1]
	\Require {$A,B,C,Z_0,V_k,H,Q,R,R_{\rm d}$}
		\State Create interpolated lifting function \(\hat{\Phi}\)
		according to \ref{sec:interpolation_phi}.
		\State Decide on the strategy for inverting \(\Psi\) as 
		in \ref{sec:invertibility_psi}, optionally include the cost \eqref{eq:monotonicity} into the term \(\theta(\cdotp)\) in \eqref{eq:problem_description}.
		\State Decide on the strategy for dealing with the input rate bounds, if applicable, according to \ref{sec:inputrate}.
		\State Formulate the KMPC, either in the sparse \eqref{eq:MPC2QP} or the dense \eqref{eq:MPCasQP} formulation.
		\Ensure function $\textrm{KMPC}(y_{\rm ref},x_0,u_{\rm prev}) \rightarrow v^{\star}$
	% \EndProcedure
	\end{algorithmic}
\end{algorithm}

The algorithm \ref{alg:mpcalgclosed} shows the usage of KMPC in closed loop. The \corplane{step 6} is equivalent to applying the control inputs to the real system.

% \MK{Je to urizly. Also write the algorithms not as a procedure, just write \textbf{Input}: $A,B,\ldots$. Same for the closed-loop algorithm.}
\begin{algorithm}[!htbp]
	\caption{Closed loop KMPC}
	\label{alg:mpcalgclosed}
	\begin{algorithmic}[1]
        \Require $y_{\rm ref},\textrm{KMPC}$
	\State Initialization: \(u_{\rm prev} \gets 0\), \(x_0 = x_{\rm init}\)
		\While{true}
			% \State Get measurement of the current state \(x_0\).
			\State \(z_0 \gets \hat{\Phi}(x_0)\).
			\State \(v_0^{\star} \gets \textrm{KMPC}(y_{\rm ref},x_0,u_{\rm prev})\) [Problem~\eqref{eq:MPC2QP}]
			\State \(u_{0}^{\star} \gets \hat{\Psi}^{-1}(v^{\star}_0) \)
			\State \(x_{0} \gets f(x_0,u_{0}^{\star})\)
			\State \(u_{\rm prev} \gets u_{0}^{\star}\)
			\State Wait for the next timestep.
		\EndWhile
	\end{algorithmic}
\end{algorithm}

}

\section{Numerical examples}
\label{sec:numerical_examples}
In this section, we shall demonstrate the following properties of our approach:
\begin{enumerate}
	\item discovering discontinuous lifting functions
	\item finding exact lift even for systems with multiple equilibria
	\item \newcor{controlling} systems with multiple equilibria
	\item controlling systems with nonlinear input functions
	\item control of \newcor{realistic}, highly nonlinear systems which are difficult to 
	control with standard methods of control.
\end{enumerate}

% can fit discontinuous lifting functions
% can control systems with multiple equilibria
% can control highly nonlinear systems, which are difficult to 
% control with standard methods of control.

The first two properties are demonstrated in the first example,
on a system with known analytical solution.
The discontinuous lifting is of particular interested,
since a lot of current methods (such as EDMD and its derivatives)
 require the 
prior knowledge (or guess) of the lifting function(s) \(\Phi\).
In order to find the Koopman operator using these methods, we would 
need to know whether its lifting functions are discontinuous
(along with the particular type \corplane{and location} of discontinuity).
This is not required by our method.

The second example shows the control of a system with multiple equilibria, using the KMPC, we compare our method to EDMD \cite{Korda_koop} and Optimal eigenfunction \cite{Korda_opt_g}.

The third example shows control of a system that cannot be controlled
by other Koopman methods which do not have nonlinear input transformation.

The last example is of a more practical nature; 
\newcor{it shows that our approach can be used to synthesize a single controller that can 
drive the vehicle in normal conditions as well as to 
stabilize it
from an unstable state.
Vehicles are highly nonlinear systems, especially when the wheels lose grip
with the road. 
Therefore such maneuvers are rather challenging, since they 
require exploitation of the nonlinear dynamics of the vehicle.
We compare both maneuvers againts MPC based on local linearization (LMPC)
and a Nonlinear MPC (NMPC).}
% it 
% shows that our approach can be used to bring a vehicle 
% into a drift and stabilize it there. 
% Vehicles are highly nonlinear when the wheels lose grip with the road, making this maneuver very difficult since it 
% requires the exploitation of the nonlinear dynamics of the vehicle. We compare the results with linear and non-linear MPC.

\subsection{Discontinuous lifting with multiple equilibria}
This example demonstrates that our algorithm can \newcor{find} 
discontinuous lifting functions and systems with multiple equilibria.
We will use the system from \cite{Bakker2019}, 
to which the analytical solution is known.
We empirically show that our algorithm converges to the analytical solution from any initial \newcor{condition of the ADAM solver}.

Consider the nonlinear system 
\begin{align}
\label{eq:mult_eq_def}
\begin{split}
\dot{x} &= f(x)\\
f(x) &= \begin{cases}
	-(x-1) & \text{if } x > 0\\
	0 & \text{if } x = 0\\
	-(x+1) & \text{if } x < 0
\end{cases}
\end{split}
\end{align}
\newcor{The corresponding Koopman operator derived in \cite{Bakker2019} 
is of the form 
\begin{equation}
    \dot{z} = Az, y = Cz,
\end{equation}}
where
\begin{align}
\begin{split}
	\label{eq:koop_iso}
A &= \begin{bmatrix}
	-1 & 1\\0 &0
\end{bmatrix} \\ 
C &= \begin{bmatrix}
	1 & 1
\end{bmatrix}.
\end{split}
\end{align}
The output \newcor{\(y = z_1+z_2\)} is equal to the nonlinear state \(x\).
The eigenvalues and eigenfunctions of the operator \eqref{eq:koop_iso}
are
\begin{align}
\begin{split}
\lambda_1 = -1, &\quad \Phi_1 = x - \text{sign}(x), \\ 
\lambda_2 = 0, &\quad \Phi_2 = \text{sign}(x).
\end{split}
\end{align}

\begin{figure}[!htpb]
	\centering
	\begin{flushright}
		% 	% \pgfplotsset{every axis post/.style={axis y line*=right,ylabel near ticks, yticklabel pos=right}}
		% \pgfplotsset{ /tikz/every picture/.append style={trim axis bottom}}
		\inputnamedtex{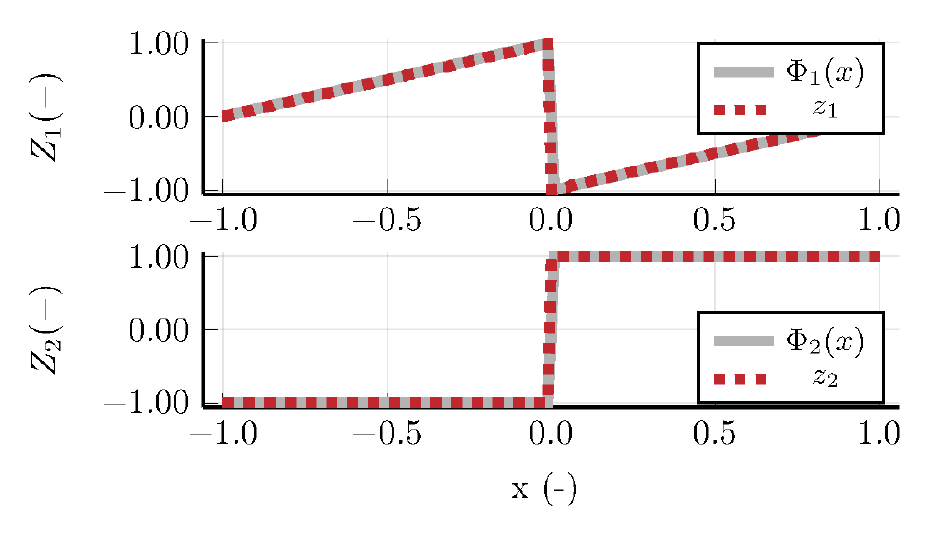}
		\end{flushright}
	\caption{Comparison of analytical and learned lifted function of the one-dimensional 
	system with multiple isolated points. We use \(z_1\) and \(z_2\) to denote the 
	first and second coordinate of the lifted space.}
		\label{fig:1d_iso}
\end{figure}
\newcor{
In order to learn the Koopman predictor, 
we have generated 300 trajectories of the system \eqref{eq:mult_eq_def} of length \(H_{\rm T} = 20\) which were sampled with
\(T_{\rm s} = 0.1s\). 
Using the approach \ref{alg:koopalg}, we obtained a discrete-time Koopman predictor $(\hat{A}_{\rm d},\hat{C}_{\rm d})$, which was transformed into a continuous system by \((\hat{A},\hat{C}) = (\frac{\log(\hat{A}_{\rm d})}{T_{\rm s}},\hat{C_{\rm d}})\)}.
To compare our result with the analytical solution \eqref{eq:koop_iso}, we transformed both the analytical system \((A,C)\)
and the learned approximation \((\hat{A},\hat{C})\) into observer canonical form 
\begin{alignat*}{3}
% \begin{split}
A_{\rm o} &= \begin{bmatrix}
	0 & 1 \\ 
	0 & -1
\end{bmatrix}, & C_{\rm o} = \begin{bmatrix}
	1 & 0
\end{bmatrix}, \\ 
\hat{A}_{\rm o} &= \begin{bmatrix}
	-0.0003 &  1.0508 \\ 
	0.005 &  -0.9989
\end{bmatrix}, & \hat{C}_{\rm o} = \begin{bmatrix}
	1 & 0
\end{bmatrix}.
% \end{split}
\end{alignat*}
We can see that the approximated system is numerically close 
to the analytical solution.
The continuous eigenvalues of learned Koopman 
system were \( \hat{\lambda}_{1,2} = \begin{bmatrix}
	-1.004& 0.005
\end{bmatrix}\). The estimated eigenfunctions are compared to the 
real ones in Fig.\ref{fig:1d_iso}. 

\vclast{
We compare the open-loop prediction capabilities of our method with EDMD
of different orders $n_z$ in Fig.\ref{fig:1d_edmd}. The basis functions for the EDMD
were thin plate spline radial basis functions
\begin{equation}
\label{eq:rbf}
    \Phi_{\rm rbf}(x) = || x - x_{\rm c} ||^2 \textrm{log}(||x - x_{\rm c} ||),
\end{equation}
with centers $x_{\rm c}$ selected randomly from the interval $[-1,1]$.
We can see that the proposed method provides more precision than EDMD while 
having only two lifted states.
\begin{figure}[!htb]
	\centering
		\begin{flushright}
		% 	% \pgfplotsset{every axis post/.style={axis y line*=right,ylabel near ticks, yticklabel pos=right}}
		\inputnamedtex{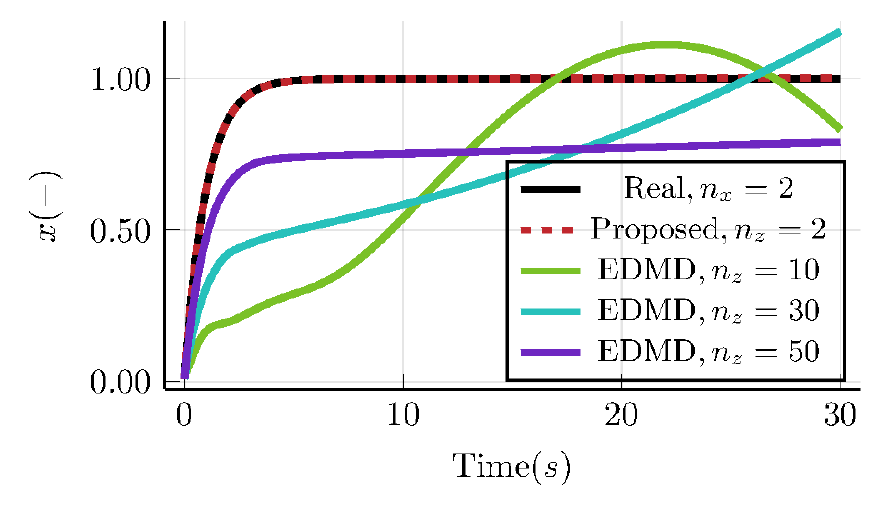}
		\end{flushright}
		\caption{Open-loop predictions of our method and EDMD of different orders.
  The initial state of the trajectories was $x_0 = 0.01$, close to the unstable equilibrium
  at $x_{\rm e} = 0$.
  The proposed predictor has the lowest order and the highest precision due to 
  the precise identification of the discontinuous lifting functions.}
		\label{fig:1d_edmd}
\end{figure}
}

To test convergence of our algorithm, we learned 100 
Koopman predictors with different initial conditions \COR{of the ADAM solver} generated 
according to \ref{table:init}. 
% \MK{Explain that you talk about the iterations of theA ADAM solver.}
The eigenvalues converged to the analytical result in all cases.
The results can be seen in Fig.\ref{fig:1d_convergence}.

\begin{figure}[!htb]
	\centering
		\begin{flushright}
		% 	% \pgfplotsset{every axis post/.style={axis y line*=right,ylabel near ticks, yticklabel pos=right}}
		\inputnamedtex{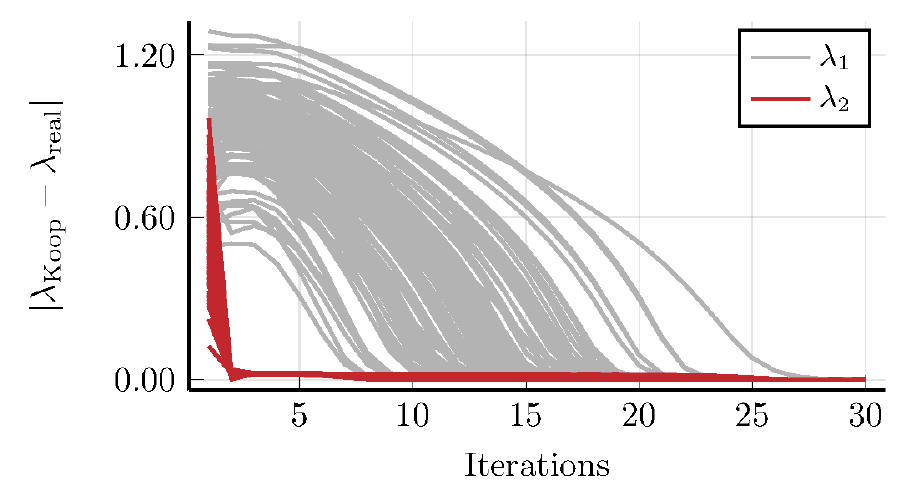}
		\end{flushright}
		\caption{Evolution of the error of the \newcor{continuous} Koopman eigenvalues during the solver iterations for 100 different initial 
		conditions. 
		Since the optimization is always done in discrete time,
		the errors were also calculated for the discrete eigenvalues.}
		\label{fig:1d_convergence}
\end{figure}

We can see that our approach was able to discover discontinuous 
eigenfunctions and system with multiple equilibria  without any prior information. 
% \rednote{TODO napsat nekam vys (intro?) ze prior info muzu vlozit inicializaci 
% tech opt. promennych} \MK{Nekde si to zapamatuj a napis pred submission do journalu.}

\subsection{Control of a system with multiple equilibria}
In this example, we deal with the damped Duffing oscillator with forcing. We will show that we can steer the system into all
of its equilibria, \COR{including the unstable one}, with the KMPC.
This maneuver is also possible with other methods 
for approximating the Koopman predictor, we shall 
therefore provide a comparison with them, namely EDMD and 
the Optimal eigenfunction approach from \cite{Korda_opt_g}.

The continuous dynamics are
\begin{align}
\begin{split}
	\label{eq:duffing}
	\dot{x}_1 &= x_2, \\ 
	\dot{x}_2 &= -0.5x_2 - x_1(4x_1^2 -1) + 0.5u,
\end{split}
\end{align}
where \(x \in [-1,1]^2\) and \(u \in [-1,1]\).
The system has 3 equilibria \(x_{\textrm{e}_1} = [-0.5,0]\), \(x_{\textrm{e}_2} = [0,0]\),
and \(x_{\textrm{e}_3} = [0.5,0]\).

For training, we discretize the system \eqref{eq:duffing} using Runge-Kutta 4 with sampling time  \(T_{\rm s} = 0.02 \rm s\).
 We used \(N=100\) trajectories, each containing 1000 samples.
 The trajectories were split into shorter ones with length \(H_{\rm T} = 20\). 
%  The used dataset is indicated in Fig.\ref{fig:duffing_data}.
The lifted space has size \(n_z = 30\) and the input was 
quantized equidistantly with \newcor{11 quantization levels.}
All of the three methods used the same data, only the Optimal eigenfunctions
were learned directly on the long trajectories (since it is benefitial for the method), and the EDMD was learned on pairs of consecutive states.
The lifting functions for the EDMD were thin plate splines used in \cite{Korda_koop}. 
The dataset had the control-separating properties discussed in 
\ref{sec:trajectory_preparation} and proven in the Lemma \ref{sec:dataset_control_separation}, the EDMD also profited from this as it 
was able to control the system as well as the other two considered methods.

The MPC parameters were \(Q = \rm{Diag}(15,0.1)\), \(R = 1\), and \(H = 60\) (1.2s), \(y_{\rm [up,low]} = \pm[1,1]\).
The original input bounds were \(\Delta u_{\rm [up,low]} = \pm 46\) and \(u_{\rm [up,low]} = \pm 1\).
\COR{The lifted input bounds were \(\Delta v_{\rm [up,low]} = \pm 5\) and  \(v_{\rm [up,low]} = [ 0.108, -0.11]\).}
Note that the bounds for the lifted inputs depend on the
particular initialization of the variables \eqref{eq:problem_description}. 
In our case, the function \(\Psi\) was \emph{linear}
after the optimization \eqref{eq:problem_description}, approximately 
\(\Psi(u) \dot{=} 0.1u\).

The Figure \ref{fig:duffing_xu} shows the three methods on a maneuver 
which requires to stabilize the system at all three equilibra and the point \([0.25,0]^\top\), which 
is not an equilibrium. We can see that the EDMD has issues with stabilizing the system at the non-equlibrium point.
The Figure \ref{fig:duffing_short} shows that the 
methods are have comparable results apart from the non-equilibrium point.

The Figure \ref{fig:duffing_long} shows the same maneuver with a longer 
prediction horizon which usually leads
 to better closed-loop performance.
 We see that the EDMD and optimal eigenfunctions
introduce oscillations into the system, making the performance worse,
while the proposed method improved its behaviour by smoothening out 
its overshoots (most visibly around 4s and 11s).

\begin{figure}[!htb]
	\centering
	\begin{flushright}
	\subfloat[Results for short horizon \(H = 20\).
	All methods are able to control the system, although EDMD has 
	large overshoots and oscillations.]
	{
	\inputnamedtex{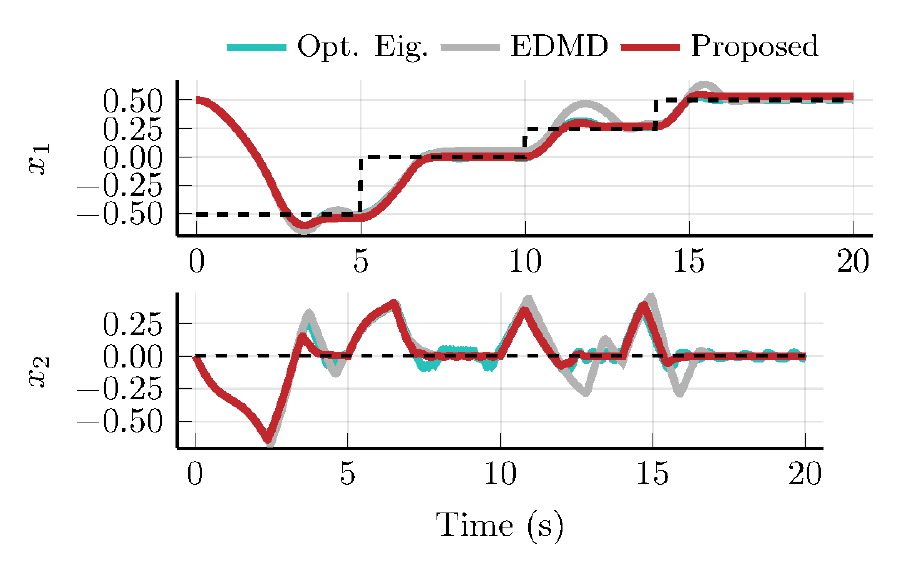}
		\label{fig:duffing_short}}
	\end{flushright}
% \end{figure}
% \begin{figure}[!htb]
	\centering
	\begin{flushright}
	\subfloat[Comparison with longer horizon \(H=60\).
	Both EDMD and Op. Eig. oscillate around the equilibria.]{
	\inputnamedtex{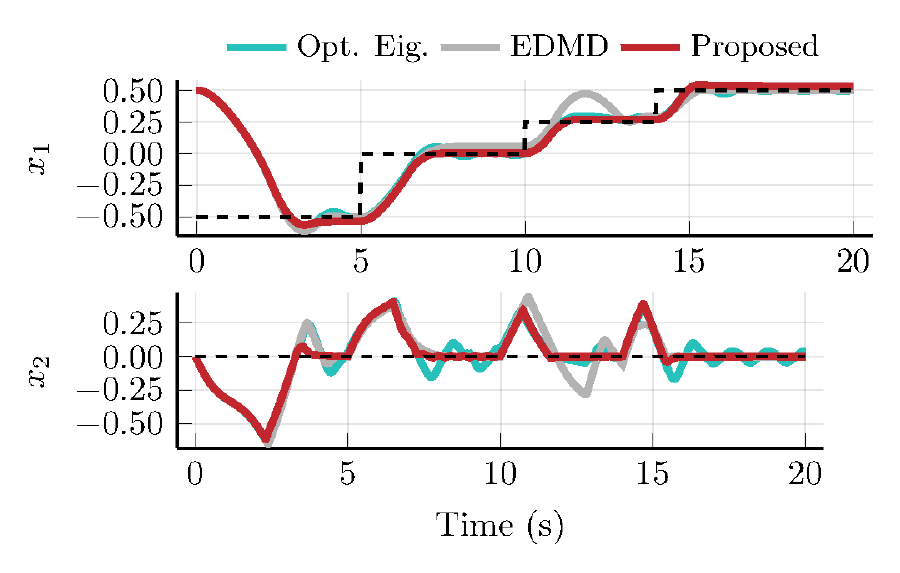}
		\label{fig:duffing_long}}
	\end{flushright}
		\caption{Comparison of the proposed method, EDMD, and the optimal eigenfunction method. The maneuver visited all three equilibria 
		as well as the state \([0.25,0]^\top\), which is not an equilibrium.
		The EDMD was the least stable, whereas the optimal eigenfunctions 
		were destabilized only with a large prediction horizon.
		The proposed method provided the same results for all }
		\label{fig:duffing_xu}
\end{figure}
% {\color{red}TODO grafy napsat o EDMD
% }

\subsection{Control with nonlinear input function}
In this example, we use the Duffing oscillator with changed control term.
The continuous dynamics are
\begin{align}
\begin{split}
	\label{eq:duffing_usquared}
	\dot{x}_1 &= x_2, \\ 
	\dot{x}_2 &= -0.5x_2 - x_1(4x_1^2 -1) - 0.5u^2,
\end{split}
\end{align}
where the square in the control input is the only change from the previous example.
The dataset parameters and the MPC setup are exactly the same as in 
the previous case. The point of this example is to show that 
predictors that use the original input \(u\) cannot approximate and therefore 
control certain class of systems, such as \eqref{eq:duffing_usquared}.

We chose EDMD as a representant of the predictors with original control 
input and learned it alongside of our method on the system \eqref{eq:duffing_usquared}.
Our method resulted in a predictor with the input lifting function 
shown in Fig. \ref{fig:duffing_usquared_input}, which is simply a scaled (and shifted) 
\(- 0.5u^2\).
We attempted to perform a maneuver that would bring the Duffing oscillator 
from  one stable equilibrium to the other, i.e. from \([0.5,0]^\top\) to \([-0.5,0]^\top\). The results can be seen in Fig.\ref{fig:duffing_usquared}.
We see that EDMD was not able to leave the equilibrium.
The Figure \ref{fig:duffing_usquared_2d} shows the maneuver in 
a state space plot. We see that our controller exploited the 
dynamics and did not apply any control when the system was going 
to the equilibrium on its own (at acceptable rate, we see that the controller
sped up the converge near the equilibrium).

% Note that the outcomes of EDMD hold for every
% Koopman predictor that does not transform the system input; 
% it is not bound to EDMD but rather to the predcitor structure
% that has \(\Psi(u) = u\).  \MK{The previous sentecen is a very strong statement. Either erase or prove or qualify.}

\begin{figure}[!ht]
	\centering
	\begin{flushright}
		% \pgfplotsset{every axis post/.style={axis y line*=right,ylabel near ticks, yticklabel pos=right}}
	\inputnamedtex{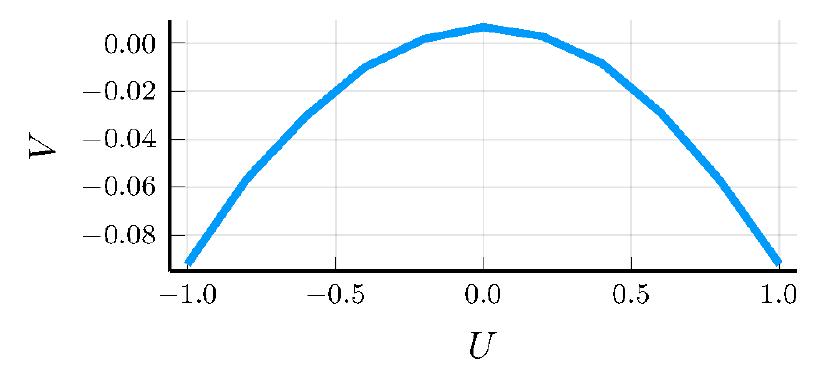}
	\end{flushright}
		\caption{Learned input transformation of a Duffing oscillator with 
		nonlinear control term \(-0.5 u^2\).}
		\label{fig:duffing_usquared_input}
\end{figure}

\begin{figure}[!ht]
	\centering
	\begin{flushright}
		% \pgfplotsset{every axis post/.style={axis y line*=right,ylabel near ticks, yticklabel pos=right}}
	\inputnamedtex{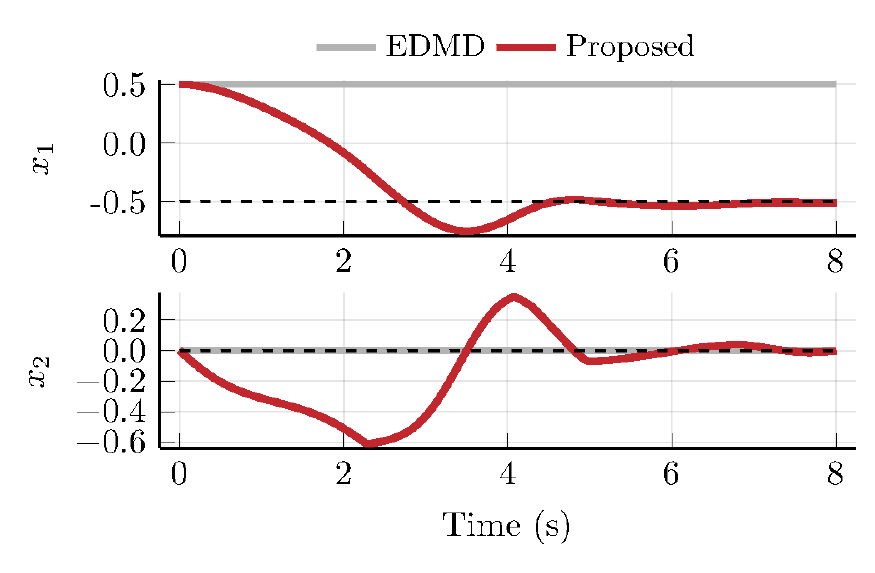}
	\end{flushright}
		\caption{Control of a Duffing oscillator with nonlinear control 
		term. EDMD was not able to leave the equilibrium because the method 
		does not consider nonlinear input transformations.}
		\label{fig:duffing_usquared}
\end{figure}
\begin{figure}[!ht]
	\centering
	\begin{flushright}
		% \pgfplotsset{every axis post/.style={axis y line*=right,ylabel near ticks, yticklabel pos=right}}
	\inputnamedtex{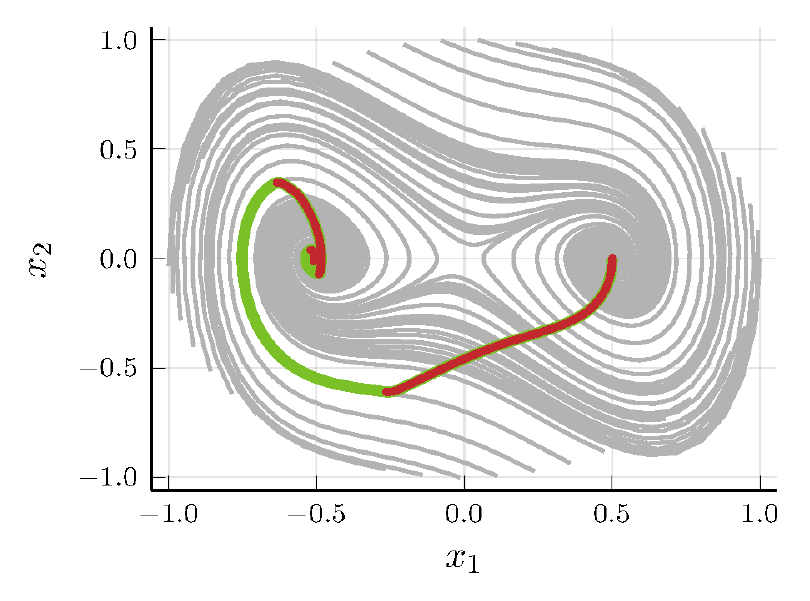}
	\end{flushright}
		\caption{Nonlinear Duffing control trajectory (red/green) plotted over 
		autonomous trajectories (grey) of the system. Control was applied 
		only in the red parts of the trajectory.}
		\label{fig:duffing_usquared_2d}
\end{figure}
%TESTBLABLA
% \begin{figure}[!htb]
% 	\centering
% 	\includegraphics[width=0.5\textwidth]{duffing_xu.pdf}
% 		\caption{Duffing oscillator controlled by Koopman MPC - state and control trajectories.}
% 		\label{fig:duffing_xu}
% \end{figure}

\subsection{Singletrack vehicle model}

% \MK{Nejak tady  u toho auta nevidim situaci, kde Koopman prevalcuje LMPC. Nebo jsem neco nepochopil?}

This example will show the algorithm on a singletrack vehicle model
used in \cite{Cibulka2021}. 
We shall present open-loop prediction capabilities, 
show that the lifted variables may have real-world physical meaning, and demonstrate the advantage of using the KMPC over a local-linearization-based MPC (we shall call it Linear MPC from now on).
\\
\subsubsection{Model description}
Due to the \newcor{complex} nature of the model, we shall 
present it only as a black-box model and refer the reader to 
\cite[Section 2]{Cibulka2021} for full model derivation. The vehicle is modeled as a planar singletrack model, also referred to as bicycle model. The main source of nonlinearities
are the tires, which are modeled using the 
high-fidelity "Pacejka tire model (2012)". This tire model contains over a hundred of parameters, which are measured on the physical tire itself. The parameter set used in this work is from the \textit{Automotive Challenge 2018} organized by Rimac Automobili.
We refer the reader to \cite[Chapter 4]{Pacejka2012} for details on the full tire model.

For our purposes, the vehicle is a nonlinear function 
\begin{equation}
	\label{eq:vehicle_nlfun}
	\begin{bmatrix}
		\dot{v}_x \\ \dot{v}_y \\ \dot{r}
	\end{bmatrix}
	= f_{\rm v}\left(\begin{bmatrix}
		v_x \\ v_y \\ r
	\end{bmatrix},\begin{bmatrix}
		\delta \\ \lambda
	\end{bmatrix}\right),
\end{equation}
where the states \(v_x ({\rm m/s}),v_y({\rm m/s})\), and \(r({\rm rad/s})\) are 
the vehicle longitudinal, lateral, and angular velocities in this order.
The inputs \(\delta(^{\circ})\) and \(\lambda(-)\) are the front steering angle and the rear-tire slip ratio respectively.
The bounds on the states and inputs are 
\(v_x \in \left[\newcor{0},30\right] {\rm m/s} \), \(v_y \in \left[-30,30\right] {\rm m/s} \), \(r \in \left[-15,15\right] {\rm rad/s}\), \(\delta \in \left[-30,30\right]^{\circ}\), and \(\lambda \in \left[-1,1\right]\). 

The model was discretized using Runge-Kutta 4 with sample time \(T_{\rm s} = 0.02s\).
\\
\subsubsection{Predictor learning}
The number of learning trajectories was \(N = 8384\) with
length \(H_{\rm T} = 10\) (\(0.2\)s).
A state was considered feasible if its kinetic energy 
was less that 300kJ, forward velocity \(v_x\) was positive, and the front-wheel slip angles were less than \(15^\circ\). 
A trajectory was considered feasible if \(70\%\) of its 
states were feasible.

The dimension of the lifted state-space was \(n_z = 45\).
The vehicle has the following symmetry
\begin{equation}
	\label{eq:vehicle_nlfun_sym}
	\begin{bmatrix}
		\dot{v}_x \\ -\dot{v}_y \\ -\dot{r}
	\end{bmatrix}
	= f_{\rm v}\left(\begin{bmatrix}
		v_x \\ -v_y \\ -r
	\end{bmatrix},\begin{bmatrix}
		-\delta \\ \lambda
	\end{bmatrix}\right),
\end{equation}
which was exploited by the proposed predictor according to the 
Section \ref{sec:symmetry}.

% The total kinetic energy of the vehicle was limited 
% to 300kJ, which is equivalent to a forward-driving car at 77km/h.

% The initial conditions of the trajectories \(\mathcal{T}_i\) were all forward-driving 
% states (\(v_x > 0\)). Furthermore, 70 \% of the states in each trajectory were required to have \(v_x > 0\). 
% The length of the trajectories was \(H_{\rm T} = 10\) (0.2s).

% Trajectories where the slip angle of the front wheel
%  exceeded \(15^\circ\) were excluded.
% This condition is meant to exclude states where the front wheel loses traction and the impact of the steering input diminishes.
% \COR{
% The number of learning trajectories was \(N = 8384\).}
% \\

% The final number of trajectories was \(N = 8384\) but
% since the dataset was symmetric and 
% half of the initial \MK{???} states were the identical (see \ref{sec:symmetry} and \ref{sec:trajectory_preparation}), the number of unique lifted states \(z_0^i\) was only 2096.
% The resulting dataset is depicted in the Fig. \ref{fig:veh_data}.
% \\

% \begin{figure}[!htb]
% 	\centering
% 	\includegraphics[width=0.5\textwidth]{vehicle_data.pdf}
% 		\caption{8\% of the dataset used for learning the Koopman predictor
% 		for the Singletrack vehicle model.
% 		Blue depicts autonomous trajectories and red is for 
% 		controlled trajectories.
% 		The purple cirles are the initial conditions of each trajectory.
% 		The dataset was limited to include only forward-driving states.
% 		The conic shape is caused by exclusion of states in which 
% 		is the influence of control mitigated.
% 		}
% 		\label{fig:veh_data}
% \end{figure}

\subsubsection{Analysis of the Koopman predictor}
Before presenting results of closed-loop control, we would like to show properties of the approximated Koopman operator.

The open-loop predictions remain accurate far longer than
the \(H_{\rm T} = 10\) used for learning.
The Figures \ref{fig:car_openloop_learned_state} and \COR{\ref{fig:car_openloop}}
show trajectories of length \(H_{\rm ol} = 100\); the former uses a state from \(Z_0\) as an initial condition, whereas the latter uses a state with \newcor{non-negligible} lifting error. Both examples show good prediction and asymptotic properties, considering the fact that they are 10 times longer than the learning trajectories.

Another notable result are the lifting functions in the Figures \ref{fig:veh_lift_1}
and \ref{fig:veh_lift_2}: they appear to have a physical meaning. 
The lifting function
of the steering angle (Fig.\ref{fig:veh_lift_1}) has a shape of a lateral force characteristics under constant acceleration,
 and the lifting function of the rear slip ratio (Fig.\ref{fig:veh_lift_2}) has the typical shape of a longitudinal force characteristics.
This means that both lifted inputs can be thought of as scaled and shifted forces, \newcor{which are a major nonlinearity present in the vehicle model.}
\begin{figure}[!ht]
	\centering
	\begin{flushright}
		% \pgfplotsset{every axis post/.style={axis y line*=right,ylabel near ticks, yticklabel pos=right}}
	\inputnamedtex{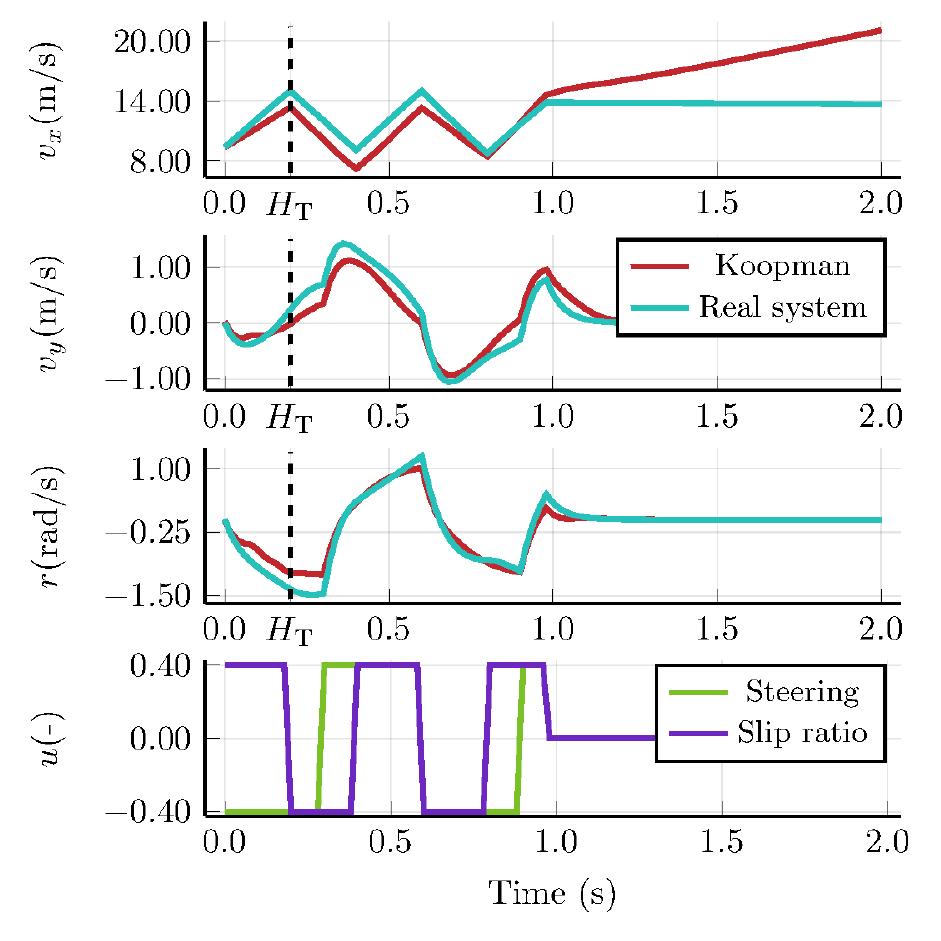}
	\end{flushright}
		\caption{Open-loop prediction with square-wave forcing. The inputs are normalized. \newcor{The black dashed line shows the learning horizon \(H_{\rm T}\).
		The trajectory is 10 times longer than the learning trajectories \(\mathcal{T}_i\).}
		The initial state is taken from the set \(Z_0\) used for learning.
		\COR{The second half of the trajectory shows the asymptotic properties of the predictor;
		 \newcor{the longitudinal velocity slowly increases, we 
		 attribute this to the fact that the predictor is always actuated in the lifted state space (see Fig.\ref{fig:veh_lift_2}).}
		 }}
		\label{fig:car_openloop_learned_state}
\end{figure}
\begin{figure}[!htb]
			\begin{flushright}
				% \pgfplotsset{every axis post/.style={axis y line*=right,ylabel near ticks, yticklabel pos=right}}
			\inputnamedtex{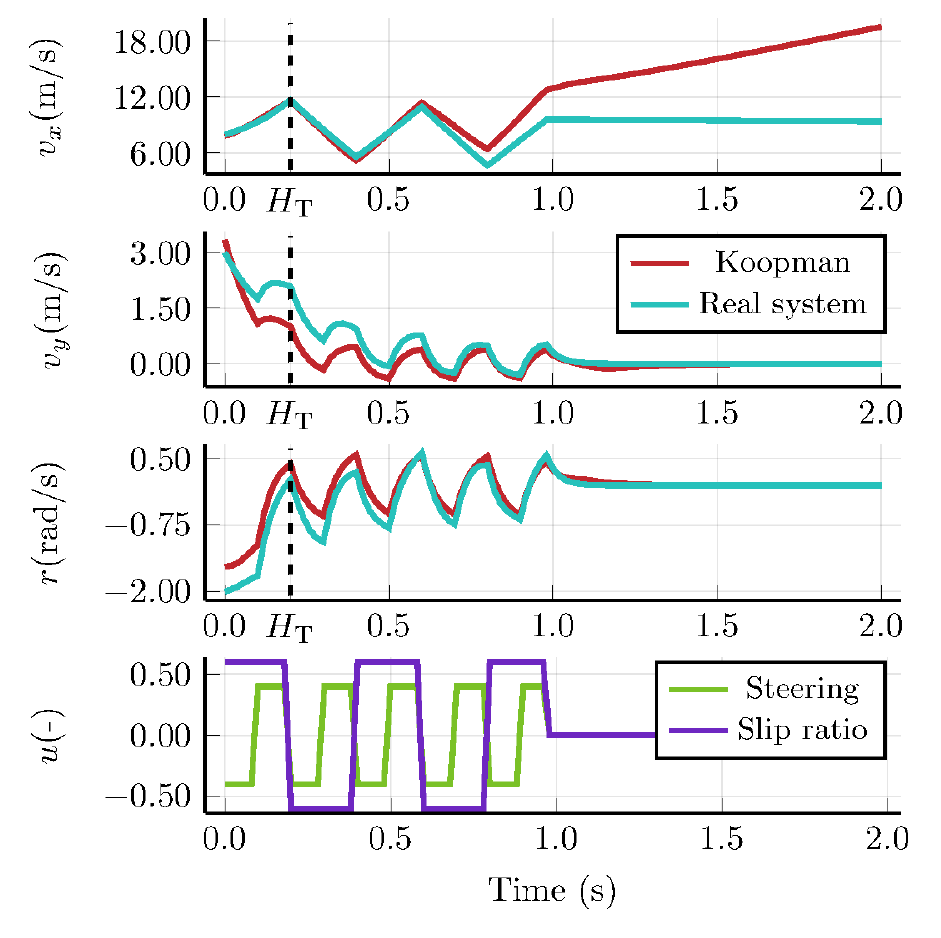}
			\end{flushright}
	\caption{Open-loop prediction with square-wave forcing. The inputs are normalized.
	\newcor{The black dashed line shows the learning horizon \(H_{\rm T}\).}
		The trajectory is 10 times longer than the learning trajectories \(\mathcal{T}_i\).
		The chosen initial state has noticeable lifting error (the trajectories starts with an offset); in spite of that, the trend of the predicted trajectory is similar to the real one. 
		\COR{The asymptotic properties are the 
		same as in Fig.\ref{fig:car_openloop_learned_state}.}
		}
		\label{fig:car_openloop}
\end{figure}
\begin{figure}[!ht]
	\centering
	\begin{flushright}
		\inputnamedtex{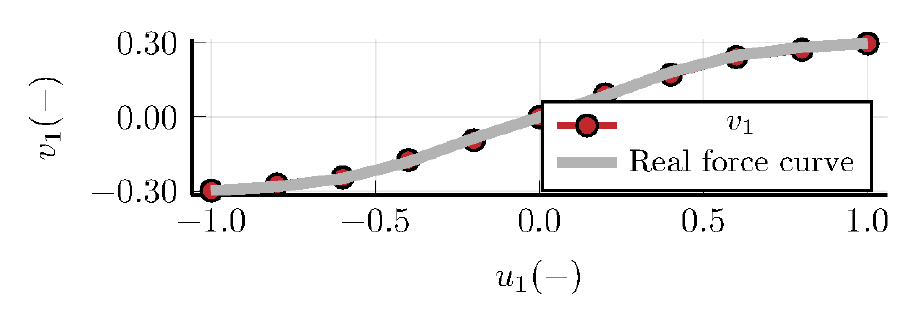}
	\end{flushright}
\caption{Lifting function of the (normalized)
steering angle, compared to scaled lateral force of the front tire at 
\(v_x = 10 {\rm m/s}\) with slip ratio \(\lambda = 0.2\).}
\label{fig:veh_lift_1}
\end{figure}

\begin{figure}[!ht]
	\centering
	\begin{flushright}
		\inputnamedtex{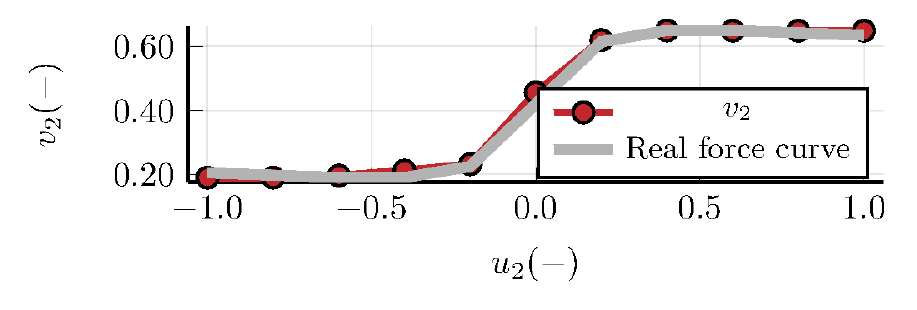}
	\end{flushright}
\caption{Lifting function of the rear slip ratio compared to a 
scaled and shifted real force curve of the rear tire for the vehicle driving forward at \(v_x = 10 {\rm m/s}\). \newcor{Note that the zero of the lifted input is shifted and the 
lifted input is always positive. We attribute this to the fact 
that the learning data contained mostly forward-driving states}.
}
\label{fig:veh_lift_2}
\end{figure}

\subsubsection{Control}
In this section, we compare the control performance of the proposed KMPC
against EDMD, Linear MPC, and Nonlinear MPC.
The operating point of the Linear MPC is \(x_{\rm op} = \begin{bmatrix}
	16.7 &0 & 0
\end{bmatrix}^\top, u_{\rm op} = \begin{bmatrix}
	0 & 0 
\end{bmatrix}^\top\) unless said otherwise.
\corr{The Nonlinear MPC implements the problem \eqref{eq:mpc_nl}, with all the bounds and 
the cost function identical to KMPC, except for the input weight \(R\), as explained below in \eqref{eq:R_uv}.}

The EDMD had the same lifted state space dimension \(n_z=45\) and the
lifting was done using thin plate spline radial basis functions \cite{Korda_koop}. 
The learning dataset was the same as for the proposed method,
only scaled to the unit box, as was the case in \cite{Korda_koop}
where the EDMD-based KMPC was first introduced.

% We show that the KMPC is able to control the vehicle 
% during standard maneuvers as well as stabilize the vehicle from 
% nonlinear states with loss of traction.

In all examples, the MPC parameters were \(Q = \rm{Diag}(1,1,1)\), \(R_v = \rm{Diag}(300,30)\), and \(H = 30\) (0.6s). 
The bounds on outputs and inputs were set according to the model description  \eqref{eq:vehicle_nlfun}. 

The input rate constraints were not used in order to make
 the controllers comparable.
All MPCs had the same parametrization, except for the input-weighting matrix \(R_{u}\), which was calculated by scaling \(R_v\) as 
\begin{equation}
	\label{eq:R_uv}
	R_u = R_v \odot s s^{\top},
\end{equation}
where \(\odot\) denotes element-wise multiplication and \(s\) is a scaling vector, \(s_k = (\max(V_k) - \min(V_k)) /2\).
% \MK{Kdyz je $s$ scalar tak nepotrebujes Hadamard product}. 
\corr{The scaling was necessary to make the controllers 
comparable, since the proposed KMPC optimizes the lifted input whereas EDMD, LMPC, and NMPC
optimize the original input.}

In our previous work \cite{Cibulka2021}, we have already 
presented a Koopman-based MPC for vehicle dynamics but 
we were not able to outperform the linear controller at all times, \corplane{ only in the highly nonlinear regimes.}
For this reason, we start with a simple maneuver around an operating point 
of the Linear MPC, in order to demonstrate that our KMPC is capable of 
driving the vehicle there.

\begin{figure}[htb]
	\centering
	\begin{flushright}
		\inputnamedtex{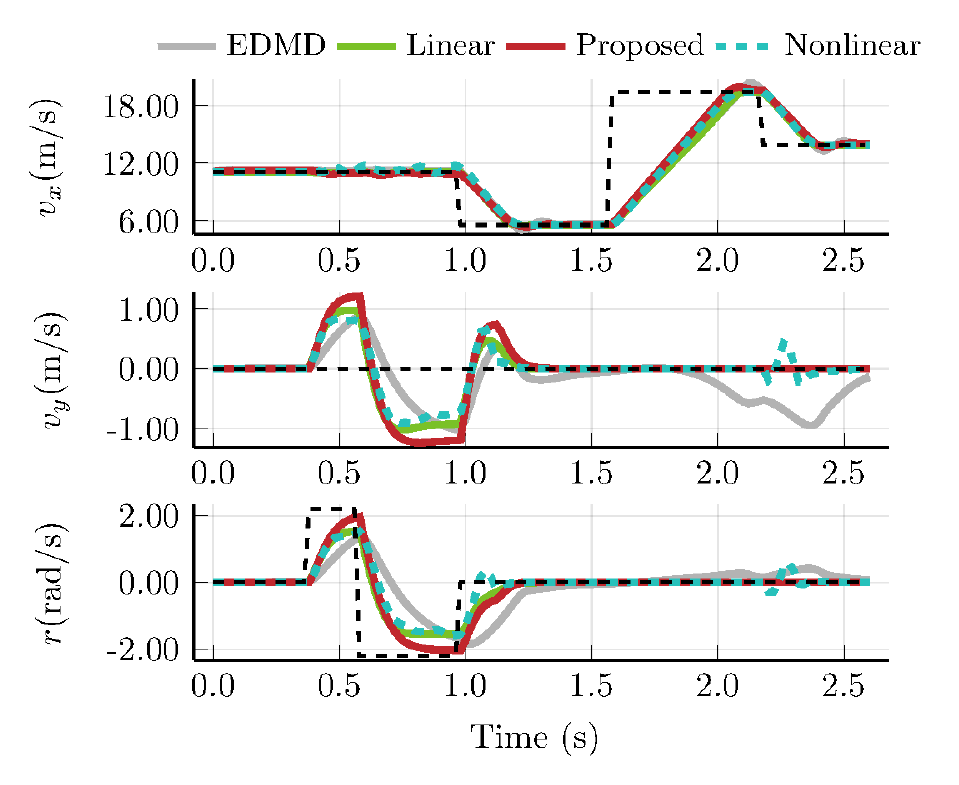}
		\end{flushright}
		\caption{
			Simple maneuvers around the operating point of the linear controller. The turn radius was 5 meters. Black dotted lines 
			are the references. We see that all controllers 
			performed similarly, although EDMD shows small oscillations 
			by the very end of the manuever.
			% Note the different offsets from the 
		}
		\label{fig:normal_maneuver}
\end{figure}

The first example in Fig. \ref{fig:normal_maneuver} shows 
a simple maneuver where the car is steered to the left, right, slows down,
and finally speeds up. 
The references are on the yawrate \(r\) and velocity \(v_x\). 
We see that all the controllers track the reference without issues.
One can notice the different offsets from the steady-state turns in the beginning of the maneuver. This is likely to be caused by the 
different weighting of the input mentioned in \eqref{eq:R_uv}.

The second example in Fig. \ref{fig:ifac_skid} is stabilization of the vehicle from a sideways skid. 
% The behaviour of both controllers is the same as in \cite{Cibulka2021}
% , the only difference being is that the maneuvers are faster, due to the missing input-rate constraints. 

% This maneuver has been described in detail in \cite[Section 6]{Cibulka2021}, therefore we will explain only very briefly its importance. \MK{Seems like you never explain it.}

The car starts in sideways skid \(x_0 = \begin{bmatrix}
	0 & -15 & 0
\end{bmatrix}^\top\) and the goal is to \newcor{perform a recovery maneuver and get the vehicle into a} forward-driving state \(x_{\rm ref} = \begin{bmatrix}
	10 & 0 & 0
\end{bmatrix}^\top\). 
The results can be seen in Fig.\ref{fig:ifac_skid}.
The Nonlinear MPC had the shortest settling time,
followed by the proposed method.
% The proposed method had the second shortest 
% \MK{Don't use the word fastest. Can be mistaken for computation time.}
%  after the Nonlinear MPC. 
 The linear controller and the EDMD managed to stabilize the system 
but with much higher cost. Notice also that its yawrate is opposite to that of NMPC, and has much higher peak.
Their longitudinal velocities were close to zero until \(0.3s\) mark, 
which delayed the stabilization.
% The EDMD was the slowest controller in this example.

 Figure \ref{fig:ifac_skidsym} shows  
 maneuver symmetrical to the one in Fig.\ref{fig:ifac_skid}, i.e. the starting state is 
\(x_0 = \begin{bmatrix}
	0 & 15 & 0
\end{bmatrix}^\top\). 
We see that the EDMD results in a very different trajectory,
with \(v_x\) very similar to the proposed method and the 
Nonlinear MPC, although the yawrate still stabilizes 
late, around the \(0.75s\) mark.
% although the dynamics of the system are symmetrical.
% Our predictor results symmetrical trajectory, since 
% the symmetries were considered 
All the other controllers resulted in trajectories symmetrical
to those in Fig.\ref{fig:ifac_skid}.

It is quite surprising and rather undersired that the 
EDMD performed so differently in two completely symmetrical maneuvers.
The proposed method remained consistent in both cases, due to the symmetry exploitation introduced in Section \ref{sec:symmetry}. 

\paragraph*{Timings}
 Table \ref{tab:mpc_times} compares the timings of the 
used controllers. We see that the Koopman, EDMD and the Linear 
MPCs have comparable computation times. 
We see that the proposed method seems to have consistent 
computation time irrespective of the maneuver, unlike the 
other two linear controllers. Nevertheless, the Linear MPC 
was the fastest, as expected.
The Nonlinear controller is not competitive in terms of timing.
\begin{table}[ht]
	\centering
	\begin{tabular}{lcccc}
	\toprule
	 & Proposed & EDMD & Linear & Nonlinear \\
	\midrule
	Fig.\ref{fig:ifac_skidsym} & 1.98ms & 1.64ms & 1.08ms & 52.8ms \\
	Fig.\ref{fig:ifac_skid}    & 1.9ms & 1.78ms & 1.04ms &  54.3ms \\
	% Fig.\ref{fig:ifac_turn} & 1.64ms & 1.65ms & 84.9ms \\
	Fig.\ref{fig:normal_maneuver} & 2.15ms & 4.67ms & 0.5ms & 69.3ms \\
	\bottomrule
	\end{tabular}
	\caption{
		Comparison of average computation times 
		for one MPC iteration. Both Koopman and Linear 
		were solved using OSQP \cite{osqp}
		using the dense MPC form found in Appendix \ref{sec:condensed}. 
		The Nonlinear MPC 
    used the multiple-shooting formulation \cite{Bock1984_multipleshootingOG,Diehl2001_multipleshooting} designed in CasADi \cite{Andersson2019} and solved via Ipopt \cite{ipopt}. All simulations were done on
		Intel Core i7-9750H CPU with 6 $\times$ 2.6GHz.
		The number of cores used by the solvers had negligible 
		effect in all cases.}
		\label{tab:mpc_times}
	\end{table}
\begin{figure}[!htb]
	\centering
	\begin{flushright}
	\inputnamedtex{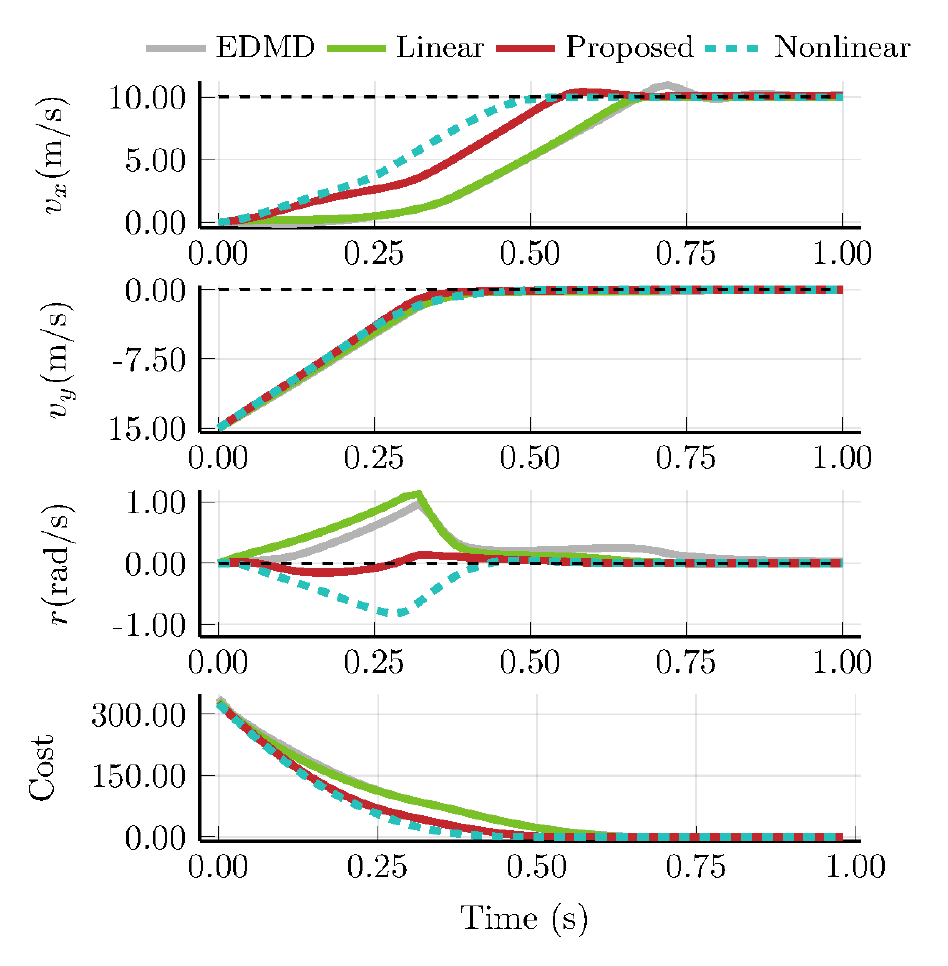}
	\end{flushright} 
		\caption{
		Stabilization of a skidding vehicle. Black dotted line is the reference. \newcor{Last graph shows the cost function the MPC problem \eqref{eq:MPC2QP}}.
		We see that the proposed KMPC and NMPC stabilized the system 
		faster than the other controllers.
		The Linear MPC performed a maneuver opposite to those 
		of the proposed method and NMPC, had much higher 
		yawrate overshoot, and resulted in significantly higher cost. 
		The EDMD did not manage to stabilize the vehicle in the given time.
		}
		\label{fig:ifac_skid}
\end{figure}
\begin{figure}[!htb]
	\centering
	\begin{flushright}
	\inputnamedtex{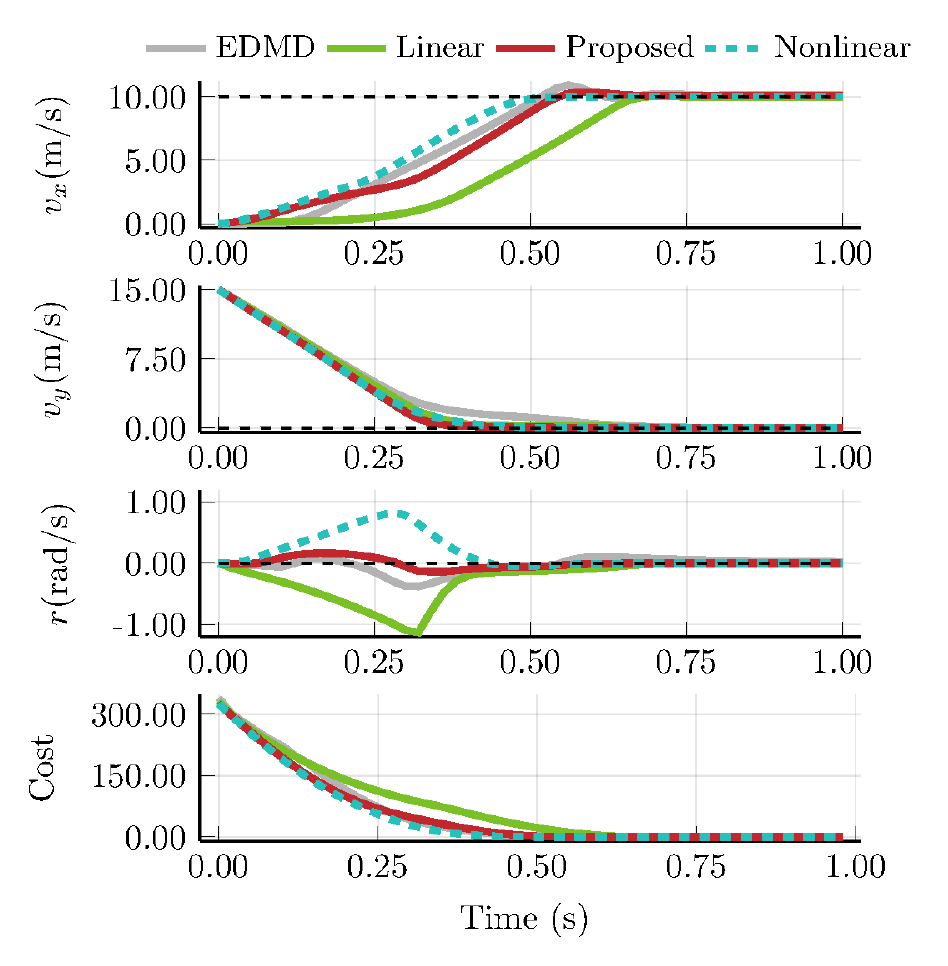}
	\end{flushright} 
		\caption{
		Symetrical maneuver to Fig \ref{fig:ifac_skid}.
		The EDMD predictor results in a different trajectory
		than in the previous case from Fig.\ref{fig:ifac_skid}, even though the problem setup 
		is symmetrical.}
		\label{fig:ifac_skidsym}
\end{figure}

\clearpage % flush figures so that they do not interrupt conclusion/references...
% \begin{samepage}
\section{Conclusion}
\label{sec:conclusion}
We presented a novel method for learning the Koopman predictor of a controlled nonlinear dynamical system.
The novelty of this approach consists in searching for 
both the Koopman predictor and the lifting functions at the same time, as well as in a nonlinear transformation of the 
control input.

Our examples show that in the classic, autonomous setting, the approach is capable of addressing 
phenomena such as multiple equilibria and discontinuous lifting 
functions. 
We demonstrate that the input transformation allows to control 
systems which are impossible to control by methods retaining the original input.

In one case, we were able to give physical meaning to the 
 learned input lifting function, without giving 
 any prior information to the 
  optimization routine.

We compared the method to two other Koopman predictors, namely EDMD
 and Optimal construction of eigenfunctions,
 to local linearization,
 and to a nonlinear controller based on CasADi \& Ipopt.  
 We show that our method outperforms the aforementioned Koopman 
 approaches, that it can approximate larger class of systems, 
 and that it is competitive to nonlinear control.
% \end{samepage}
% \iffalse

\section*{Acknowledgments}
The authors would like to thank 
Petar Bevanda for his suggestions 
regarding the presented examples.

\bibliographystyle{plain}
% \bibliography{references}
\bibliography{library_freeman_pub.bib} 

\newpage
\appendix
  % \MK{Not showing the appendix heading.}
  \section{Condensed MPC formulation}
  \label{sec:condensed}
  In this section, we show that the problem \eqref{eq:MPC2QP} can be formulated 
  in a way, that makes its size independent of \(n_{z}\).
  This is very useful for the Koopman operator, since 
  \(n_z\) will be usually large in most applications.
  We restate the problem \eqref{eq:MPC2QP} for convenience
  \begin{equation}
    \label{eq:MPC2QPap}
    \begin{aligned}
    \min_{\Delta v_t, z_t} \quad & \sum_{t=1}^{H}
    ||y_{\rm ref} - y_t||^2_{Q} + ||v_t||^2_{R} + ||\Delta v_t||^2_{R_{\rm d}} \\ 
    \textrm{s.t.} \quad & z_{t+1} = A z_t + B v_t\\
    & y_t = C z_t\\
    & v_t = \Delta v_{k} + v_{k-1}\\
    & y_{\rm low} \leq y_t \leq y_{\rm up} \\
    & v_{\rm low} \leq v_t \leq v_{\rm up} \\
    & \Delta v_{\rm low} \leq \Delta v_t \leq \Delta v_{\rm up}\\
    & z_0 = \Phi(x_0) \\ 
    & v_0 = \Psi( u_{\rm prev}),
    \end{aligned}
  \end{equation}
  where \(y_t = \begin{bmatrix}
    \hat{\chi}_t \\ \hat{x}_t
  \end{bmatrix}\), \(Q \succcurlyeq 0\), \(R \succcurlyeq 0\), and \(R_{\rm d}\succ 0\).

  The output trajectory of \eqref{eq:MPC2QPap} can be written as
  % \begin{equation}
  % 	\begin{bmatrix}
  % 		y_1 \\y_2 \\ \vdots \\ y_H
  % 	\end{bmatrix}
  % 	=
  % 	\begin{bmatrix}
  % 		CA & 0 & CB & & & \\
  % 		CA^2 & 0 & CAB & CB & & \\
  % 		\vdots & \vdots & \vdots &  & \ddots & \\
  % 		CA^H & 0 & CA^{H-1}B & \dots &  & CB \\
  % 	\end{bmatrix}
  % 	\begin{bmatrix}
  % 		z_0 \\ v_{\rm prev} \\ v_0 \\ \vdots \\ v_{H-1},
  % 	\end{bmatrix}
  % \end{equation}
  
  \begin{equation}
    \begin{bmatrix}
      y_1 \\y_2 \\ \vdots \\ y_H
    \end{bmatrix}
    =
    \begin{bmatrix}
      CB        &       &        & \\
      CAB       & CB    &        & \\
      \vdots    &       & \ddots & \\
       CA^{H-1}B & \dots &        & CB \\
    \end{bmatrix}
    \begin{bmatrix}
      v_0 \\ \vdots \\ v_{H-1}
    \end{bmatrix}
    +
    \begin{bmatrix}
        CA \\
    CA^2       \\ 
    \vdots      \\ 
    CA^H       
    \end{bmatrix}
    z_0
  \end{equation}
  or in short as
  \begin{equation}
    Y = MV + C_{z}.
  \end{equation}
  We see that the size of the matrix \(M\) is of size \(Hn_y \times Hn_v\) 
  and does not depend on \(n_z\).
  
  The input rates can be obtained as 
  % \begin{equation}
  % 	\begin{bmatrix}
  % 		\Delta v_0 \\ \Delta v_1 \\ \vdots \\ \Delta v_{H-1}
  % 	\end{bmatrix} = 
  % 	\begin{bmatrix}
  % 		0      & -I    & I  &       &        & \\
  % 		0      & 0     & -I &I      &        & \\
  % 		\vdots &\vdots &    & \ddots& \ddots & \\
  % 		0      & 0     &  \dots  &   0    & -I     &I
  % 	\end{bmatrix}
  % 	\begin{bmatrix}
  % 		z_0 \\ v_{\rm prev} \\v_0 \\ v_1 \\ \vdots \\ v_{H-1}
  % 	\end{bmatrix}
  % \end{equation}
  \begin{equation}
    \begin{bmatrix}
      \Delta v_0 \\ \Delta v_1 \\ \vdots \\ \Delta v_{H-1}
    \end{bmatrix} = 
    \begin{bmatrix}
      I  &       &        & \\
      -I &I      &        & \\
         & \ddots& \ddots & \\
       \dots  &   0    & -I     &I
    \end{bmatrix}
    \begin{bmatrix}
       v_0 \\ v_1 \\ \vdots \\ v_{H-1}
    \end{bmatrix}
    +
    \begin{bmatrix}
      -I    \\
      0     \\
       \vdots \\
      0     
    \end{bmatrix}
    v_{\rm prev} 
  \end{equation}
  or in short as
  \begin{equation}
    \Delta V = DV + C_{v}
  \end{equation}
  
  The first term in the cost function can be rewritten as 
  \begin{equation}
    \label{eq:cost1}
  \begin{alignedat}{2}
    \sum_{t=1}^{H}||y_{\rm ref} - y_t||^2_{Q} 
    =&(Y - Y_{\rm ref})^{\top} Q_H (Y-Y_{\rm ref})   \\
    =& Y^{\top} Q_H Y - 2Y_{\rm ref}^{\top}QY + Y_{\rm ref}^{\top}Q Y_{\rm ref}  \\
    =&V^{\top} (M^{\top} Q_H M ) V - 2(C_z - Y_{\rm ref})^{\top}Q_H MV + C_0
  \end{alignedat}
  \end{equation}
  where \(C_0\) contains terms that do not depend on \(V\) and \(Y_{\rm ref}\) is a column vector of \(y_{\rm ref}\) repeated 
  \(H\)-times.
  The second term is 
  \begin{equation}
    \label{eq:cost2}
    \begin{alignedat}{2}
      \sum_{t=1}^{H}||\Delta v_t||^2_{R_{\rm d}}  &= 
      \Delta V^{\top} R_{{\rm d},H} \Delta V \\ 
      &=V^{\top} ( D^{\top} R_{{\rm d},H} D ) V + 2 C_{v}{^\top}R_{{\rm d},H}DV + C_1,
    \end{alignedat}
    \end{equation}
  where \(C_1\) is a term that does not depend on \(V\).
  
  Finally
  \begin{equation}
  \label{eq:cost3}
  \begin{alignedat}{2}
    \sum_{t=1}^{H}||v_t||^2_{R}  = V^{\top} R_{H} V.
  \end{alignedat}
  \end{equation}
  The total cost can be written as 
  \begin{equation}
    V^{\top} F V + q^{\top}V,
  \end{equation}
  where 
  \begin{equation}
    \label{eq:F}
    F = M^{\top} Q_H M +  D^{\top} R_{{\rm d},H} D + R_{{\rm d},H}
  \end{equation}
  and 
  \begin{equation}
    \label{eq:q}
    q = (- 2(C_z - Y_{\rm ref})^{\top}Q_H M + 2 C_{v}^\top R_{{\rm d},H} D)^{\top}.
  \end{equation}
  The constraints can be written as 
  \begin{equation}
    \label{eq:constraints}
    \begin{bmatrix}
      M \\ -M \\ I_{H\cdotp n_v} \\ -I_{H\cdotp n_v} \\ D \\ -D
    \end{bmatrix}V
    +
    \begin{bmatrix}
      C_z \\ -C_z \\ 0 \\ 0\\ C_v \\ -C_v
    \end{bmatrix}
    \leq 
    \begin{bmatrix}
      y_{\rm up} \\ -y_{\rm low} \\ v_{\rm up} \\ -v_{\rm low} \\
      \Delta v_{\rm up} \\ -\Delta v_{\rm low}
    \end{bmatrix}
  \end{equation}
  or in short as \(G V + G_0 \leq h\).

  With the formulas \eqref{eq:F},\eqref{eq:q}, and \eqref{eq:constraints}
  we can write \eqref{eq:MPC2QPap} as the following QP
  \begin{equation}
    \label{eq:MPCasQP}
    \begin{aligned}
    \min_{ V} \quad & V^{\top} F V + q^{\top}V \\
    \textrm{s.t.} \quad & G V + G_0 \leq h.
    \end{aligned}
  \end{equation}
  % }
  % \end{appendices}

\end{document}